\documentclass[12pt,a4paper,reqno]{amsart} 
\usepackage{amssymb}
\usepackage{latexsym}
\usepackage{amsmath}
\usepackage{mathrsfs}
\usepackage[X2,T1]{fontenc}
\usepackage{cite}

\usepackage{calc}                   

\newcommand{\Char}{\operatorname{Char}}
\newcommand{\scal}[2]{\langle #1,#2\rangle}
\newcommand{\rr}[1]{\mathbf R^{#1}}

\newcommand{\nm}[2]{\Vert #1\Vert _{#2}}
\newcommand{\nmm}[1]{\Vert #1\Vert }
\newcommand{\abp}[1]{\vert #1\vert}
\newcommand{\op}{\operatorname{Op}}

\newcommand{\sets}[2]{\{ \, #1\, ;\, #2\, \} }
\newcommand{\ep}{\varepsilon}
\newcommand{\fy}{\varphi}
\newcommand{\cdo}{\, \cdot \, }
\newcommand{\supp}{\operatorname{supp}}

\newcommand{\eabs}[1]{\langle #1\rangle}     

\newcommand{\vrum}{\vspace{0.1cm}}
\newcommand{\ttbigcap}{{\textstyle{\, \bigcap \, }}}
\newcommand{\ttbigcup}{{\textstyle{\, \bigcup \, }}}
\DeclareMathOperator{\WF}{WF}

\numberwithin{equation}{section}          

\newtheorem{thm}{Theorem}
\numberwithin{thm}{section}
\newtheorem*{tom}{\rubrik}
\newcommand{\rubrik}{}
\newtheorem{prop}[thm]{Proposition}
\newtheorem{cor}[thm]{Corollary}
\newtheorem{lemma}[thm]{Lemma}

\theoremstyle{definition}

\newtheorem{defn}[thm]{Definition}
\newtheorem{example}[thm]{Example}

\theoremstyle{remark}

\newtheorem{rem}[thm]{Remark}              

\author{Stevan Pilipovi\' c}

\address{Department of Mathematics and Informatics,
University of Novi Sad, Novi Sad, Serbia}

\email{stevan.pilipovic@dmi.uns.ac.rs}

\author{Nenad Teofanov}

\address{Department of Mathematics and Informatics,
University of Novi Sad, Novi Sad, Serbia}

\email{nenad.teofanov@dmi.uns.ac.rs}

\author{Joachim Toft}

\address{Department of Mathematics and Systems Engineering,
V{\"a}xj{\"o} University, Sweden}

\email{joachim.toft@vxu.se}

\title{\textbf {Micro-local analysis with Fourier Lebesgue
spaces. Part I}}



\begin{document}

\begin{abstract}
Let $\omega ,\omega _0$ be appropriate weight functions and $q\in
[1,\infty ]$. We introduce the wave-front set, $\WF _{\mathscr
FL^q_{(\omega )}}(f)$ of $f\in \mathscr S'$ with respect to weighted
Fourier Lebesgue space $\mathscr FL^q_{(\omega )}$. We prove that
usual mapping properties for pseudo-differential operators $\op (a)$
with symbols $a$ in $S^{(\omega _0)}_{\rho ,0}$ hold for such
wave-front sets. Especially we prove that 
\begin{multline}\tag*{(*)}
\WF  _{\mathscr
FL^q_{(\omega /\omega _0)}}(\op (a)f)\subseteq \WF _{\mathscr
FL^q_{(\omega )}}(f)
\\
\subseteq \WF _{\mathscr FL^q_{(\omega
/\omega _0)}}(\op (a)f)\ttbigcup \Char (a).
\end{multline}
Here $\Char (a)$ is the set of
characteristic points of $a$.
\end{abstract}

\maketitle

\section{Introduction}\label{sec0}

\par

In this paper we introduce
wave-front sets of appropriate Banach (and Fr\'echet)
spaces. We especially consider the case when these Banach spaces
are Fourier-Lebesgue type spaces. The family of such
wave-front sets contains the wave-front sets of Sobolev type,
introduced by H{\"o}rmander in \cite {Hrm-nonlin}, as well as the
classical wave-front sets with respect to smoothness (cf. Sections 8.1
and 8.2 in \cite {Ho1}), as special cases. Roughly speaking, for any
given distribution $f$ and for appropriate Banach (or Frech\'et) space
$\mathcal B$ of tempered distributions, the wave-front set
$\WF  _{\mathcal B}(f)$ of $f$ consists of all pairs
$(x_0,\xi _0)$ in $\rr d\times (\rr d\setminus 0)$ such that no
localizations of the distribution at $x_0$ belongs to $\mathcal B$ in
the direction $\xi _0$.

\par

We also establish mapping properties for a quite general class of
pseudo-differential operators on such wave-front sets, and show that
our approach leads to a flexible micro-local analysis tools which
fits well to the most common approach developed in e.g. \cite {Ho1,
Hrm-nonlin}. Especially we prove that usual mapping properties,
which are valid for classical wave-front sets (cf. Chapters VIII and
XVIII in \cite{Ho1}) also hold for wave-front sets of
Fourier-Lebesgue type. For example, we prove (*) in the abstract,
that is, any operator $\op (a)$ to some extent shrink the wave-front
sets and the opposite embedding can be obtained by including $\Char
(a)$, the set of characteristic points of the operator symbol $a$.

\par

The symbol classes for the pseudo-differential operators are of the
form $S_{\rho ,\delta}^{(\omega _0)}(\rr {2d})$ which consists of all
smooth functions $a$ on $\rr {2d}$ such that $a/\omega _0\in S^0_{\rho
,\delta}(\rr {2d})$. Here $\rho ,\delta \in \mathbf R$ and $\omega_0$ is
an appropriate smooth function on $\rr
{2d}$. We note that $S_{\rho ,\delta}^{(\omega _0)}(\rr {2d})$ agrees
with the H{\"o}rmander class $S^r_{\rho ,\delta}(\rr {2d})$ when
$\omega _0(x, \xi )=\eabs \xi ^r$, where $r\in \mathbf R$ and $\eabs \xi
=(1+|\xi |^2)^{1/2}$.

\par

The set of characteristic points $\Char (a)$ of $a\in S_{\rho
,\delta}^{(\omega_0 )}$ depends on the choices of $\rho$, $\delta$
and $\omega_0$ (see Definition \ref{defchar}). This set is empty when
$a$ satisfies a local ellipticity condition with respect to
$\omega_0$. In contrast to Section 18.1 in \cite{Ho1}, $\Char (a)$ is
defined for all symbols in $S_{\rho ,\delta}^{(\omega_0 )}$, and not
only for polyhomogeneous symbols. Furthermore, if $a$ is a
polyhomogeneous symbol, then $\Char (a)$ is smaller than the set of
characteristic points, given by \cite{Ho1} (see Remark
\ref{compchar} and Example \ref{examplheat}). This is especially
demonstrated for a broad class of hypoelliptic partial differential
operators. For any hypoelliptic operator $\op (a)$ with constant
coefficients and with symbol $a$, we may choose the symbol class
such that it contains $a$, and such that $a$ is elliptic with
respect to that weight. Consequently, $\Char (a)$ is empty, and in
view of (*) in the abstract it follows that such hypoelliptic
operators preserve the wave-front sets, as they should (see Theorems
\ref{hypoellthm} and \ref{thmclassicWFs}, and Corollary
\ref{cormainthm23B}).

\medspace

Information on regularity in the background of wave-front sets of
Fourier Lebesgue types might be more detailed comparing to classical
wave-front sets, because we may play with the exponent $q\in
[1,\infty]$ and the weight function $\omega$ in our choice of
Fourier Lebesgue space $\mathscr FL^q_{(\omega )}(\rr d)$. By
choosing $q=1$ and $\omega (\xi )=\eabs \xi ^N$, where $N\ge 0$ is
an integer, $\mathscr FL^1_{(\omega )}(\rr d)$
locally contains $C^{N+d+1}(\rr d)$, and is contained in $C^N(\rr
d)$. Consequently, our wave-front sets can be used to investigate
micro-local properties which, in some sense, are close to
$C^N$-regularity.

\par

Another example is obtained by choosing $q=\infty$ and $\omega =\omega
_0$. If $E$ is a parametrix to a pseudo-differential operator $\op
(a)$ with $a\in S^{(\omega _0)}_{\rho ,0}$, then
$$
\op (a)E=\delta _0+\fy,
$$
which belongs locally to $\mathscr FL^\infty$, giving that $\WF
_{\mathscr 
FL^\infty}(\op (a)E)$ is empty. Hence  (*) in the abstract shows
that $\WF _{\mathscr FL^\infty _{(\omega )}}(E)$ is contained in
$\Char (a)$. In particular, if  in addition $\op (a)$
is elliptic with respect to $\omega _0$, then it follows that
$\WF _{\mathscr FL^\infty _{(\omega
)}}(E)$ is empty, or equivalently, $E$ is locally in $\mathscr
FL^\infty _{(\omega )}$. This implies that for each $x\in \rr d$ and
test function $\fy$ on $\rr d$ we have
\begin{equation}\label{huvaligen}
|\mathscr F(\fy E)(\xi )|\le C\omega (x,\xi )^{-1},
\end{equation}
for some constant $C$. Here we remark that every hypoelliptic partial
differential operator with constant coefficients is elliptic with
respect to some admissible weight $\omega _0$. Therefore, our results
can be applied in efficient ways on such operators. (See Theorem
\ref{hypoellthm} and Corollary \ref{cormainthm23B} for the details.)

\medspace

In the second part of the paper (Sections
\ref{sec6} and \ref{sec7}) we define wave-front sets with respect to
(weighted) modulation spaces (which also involve certain types of
Wiener amalgam spaces), and prove that they coincide with the
wave-front sets of
Fourier Lebesgue type. Here we also extend some wave-front
results to pseudo-differential operators with symbols which are
defined in terms of modulation spaces  of "weighted Sj{\"o}strand
type". These symbol classes  are superclasses to $S_{\rho
,0}^{(\omega _0)}(\rr {2d})$, and contain non-smooth symbols.

\par

The modulation spaces have been introduced by Feichtinger in
\cite{F1}, and the theory was developed further and generalized in
\cite{Feichtinger3,Feichtinger4,Feichtinger5,Grochenig0a}. The
modulation space $M^{p,q}_{(\omega )}(\rr d)$, where $\omega$ denotes
a weight function on phase (or time-frequency) space $\rr {2d}$,
is the set of tempered (ultra-) distributions
whose short-time Fourier transform belongs to the weighted and mixed
Lebesgue space $L^{p,q}_{(\omega )}(\rr {2d})$. It follows that
 the weight $\omega$ quantifies the degrees of
asymptotic decay and singularity of the distributions.

\par

Modulation spaces have been, in parallel,
incorporated into the calculus of pseudo-differential operators,
through the study of continuity of (classical)
pseudo-differential operators acting on modulation spaces (cf.
\cite{Tachizawa1,Czaja,Pilipovic2,Pilipovic3,Teofanov1,Teofanov2}),
as well as through the analysis of  operators of non-classical type,
where modulation spaces are used as symbol classes.
For example, after a systematic development of the modulation space
theory already had been done by Feichtinger and Gr{\"o}chenig,
Sj{\"o}strand introduced in \cite{Sjostrand1}
a superspace of $S^0_{0,0}$ which turned out to coincide with
$M^{\infty,1}$, and used this modulation space
as a symbol class. He proved that $M^{\infty,1}$ as symbol class
corresponds to an algebra of operators which are bounded on $L^2$.
Sj{\"o}strand's results were thereafter further extended in
\cite{Grochenig0,Gro-book,Grochenig1b,Toft2,To8,Toft4}.

\medspace

The paper is organized as follows. In the beginning of Section
\ref{sec1} we  recall the definition of (weighted) Fourier Lebesgue
spaces. We continue with the definition and basic properties of
pseudo-differential operators in Subsection \ref{subsec1}. Then, in
Subsection \ref{subsec2} we define sets of characteristic points for a
broad class of pseudo-differential operators and prove that these sets
might be smaller than characteristic sets in \cite{Ho1} (see also
Example \ref{examplheat} in Section \ref{sec3}). In Subsection
\ref{subsec3} we recall the definition and basic properties of
modulation spaces, and, in Subsection \ref{subsec4} we introduce a
class of pseudo-differential operators with non-smooth symbols in the
context of modulation spaces.
In Section \ref{sec2} we define wave-front sets with respect to
(weighted) Fourier Lebesgue spaces, and prove some important
properties for such  wave-front sets. Thereafter we
consider in Section \ref{sec3}  mapping properties
for pseudo-differential operators in context of these wave-front
sets, and, in particular,  we prove (*)  in the abstract.

\par

In Section \ref{sec5} we consider wave-front sets obtained from
sequences of Fourier Lebesgue space spaces. We
show that these types of wave-front sets contain classical
wave-front sets (with respect to smoothness), and that the mapping
properties for pseudo-differential operators also hold in context
of such wave-front sets. In particular, we recover the well-known
property (*) in the abstract, for usual wave-front sets
(cf. Section 18.1 in \cite{Ho1}).

\par

In Section \ref{sec6} we introduce wave-front sets with respect to
modulation spaces, and prove that they
coincide with wave-front sets of Fourier Lebesgue types. In Section
\ref{sec7} we consider mapping properties on wave-front sets of
pseudo-differential operators with symbol classes defined in terms of
weighted Sj{\"o}strand classes.

\par

Finally, we remark that the present paper is the first one in series
of papers. In the second paper \cite{PTT2} we consider products in
Fourier Lebesgue spaces, related to the new notion of wave-fronts
with applications to a class of semilinear partial differential
equations. In \cite{JPTT} the authors together with Karoline
Johansson show that the wave-front sets of Fourier Lebesgue and
modulation space types can be discretized,
and how they can be implemented in numerical computations.

\par

\section*{Acknowledgement}

\par

The authors are thankful to Professor Hans Feichtinger for his careful
reading of the manuscript, and for his constructive critisism which
lead to several improvements of the paper.

\par

In earlier versions of the paper, only the
case $\delta =0$ in $S_{\rho ,\delta}^{(\omega )}$ was
considered. The authors are grateful to Professor
Gianluca Garello who suggested to also consider the case
$\delta \geq 0$.

\par

\section{Preliminaries}\label{sec1}

\par

In this section we recall some notations and basic results. In what
follows we let $\Gamma$ denote an open cone in $\rr d\setminus 0$ with
vertex at origin. If $\xi \in \rr d\setminus 0$ is fixed, then an open
cone which contains $\xi $ is sometimes denoted by $\Gamma_\xi$.

\par

Assume that $\omega$ and $v$ are positive and measurable functions
on $\rr d$. Then $\omega$ is called $v$-moderate if
\begin{equation}\label{moderate}
\omega (x+y) \leq C\omega (x)v(y)
\end{equation}
for some constant $C$ which is independent of $x,y\in \rr d$. If $v$
in \eqref{moderate} can be chosen as a polynomial, then $\omega$ is
called polynomially moderated. We let $\mathscr P(\rr d)$ be the set
of all polynomially moderated functions on $\rr d$. If $\omega (x,\xi
)\in \mathscr P(\rr {2d})$ is constant with respect to the
$x$-variable ($\xi$-variable), then we sometimes write $\omega (\xi )$
($\omega (x)$) instead of $\omega (x,\xi )$. In this case we consider
$\omega$ as an element in $\mathscr P(\rr {2d})$ or in $\mathscr P(\rr
d)$ depending on the situation.

\par

We also need to consider classes of weight functions, related to
$\mathscr P$. More precisely, we let $\mathscr P_0(\rr d)$ be the set
of all $\omega \in \mathscr P(\rr d)\bigcap C^\infty (\rr d)$ such
that $\partial ^\alpha \omega /\omega \in L^\infty$ for all
multi-indices $\alpha$.
For each $\omega \in \mathscr P(\rr d)$, there is an equivalent
weight $\omega _0\in \mathscr P_0(\rr d)$, that is,
$C^{-1}\omega _0\le \omega \le C\omega _0$ holds for some constant
$C$ (cf. \cite[Lemma 1.2]{To8}).

\par

Assume that $\rho ,\delta \in \mathbf R$. Then we let $\mathscr
P_{\rho ,\delta}(\rr {2d})$ be the set of all $\omega (x,\xi )$ in
$\mathscr P(\rr {2d})\cap C^\infty (\rr {2d})$ such that
$$
\eabs \xi ^{\rho |\beta |-\delta |\alpha |}\frac {\partial ^\alpha
_x\partial ^\beta _\xi \omega (x,\xi )}{\omega (x,\xi )}\in L^\infty
(\rr {2d}),
$$
for every multi-indices $\alpha$ and $\beta$. Note that in contrast to
$\mathscr P_0$, we do not have an equivalence between $\mathscr
P_{\rho ,\delta}$ and $\mathscr P$ when $\rho >0$. On the other hand,
if $s\in \mathbf R$ and $\rho \in [0,1]$, then $\mathscr P_{\rho
,\delta} (\rr {2d})$ contains $\omega (x,\xi )=\eabs \xi ^s$, which
are one of the most important classes in the applications.

\medspace

The Fourier transform $\mathscr F$ is the linear and continuous
mapping on $\mathscr S'(\rr d)$ which takes the form
$$
(\mathscr Ff)(\xi )= \widehat f(\xi ) \equiv (2\pi )^{-d/2}\int _{\rr
{d}} f(x)e^{-i\scal  x\xi }\, dx
$$
when $f\in L^1(\rr d)$. The map $\mathscr F$ is a homeomorphism
on $\mathscr S'(\rr d)$ which restricts to a homeomorphism on $\mathscr
S(\rr d)$ and to a unitary operator on $L^2(\rr d)$.

\medspace

Let $q\in [1,\infty ]$ and $\omega \in \mathscr P(\rr
{d})$. The (weighted) Fourier Lebesgue space $\mathscr
FL^q_{(\omega )}(\rr d)$ is the inverse Fourier image of
$L^q _{(\omega )} (\rr d)$, i.{\,}e. $\mathscr FL^q_{(\omega )}(\rr d)$
consists of all $f\in \mathscr S'(\rr d)$ such that
\begin{equation}\label{FLnorm}
\nm f{\mathscr FL^{q}_{(\omega )}} \equiv \nm {\widehat f\cdot \omega
}{L^q} .
\end{equation}
is finite. If $\omega =1$, then the notation $\mathscr FL^q$
is used instead of $\mathscr FL^q_{(\omega )}$. We note that if
$\omega (\xi )=\eabs \xi ^s$, then $\mathscr
FL^{q}_{(\omega )}$ is the Fourier image of the Bessel potential space
$H^p_s$ (cf. \cite{BL}).

\par

Here and in what follows we use the notation $\mathscr F L^q_{(\omega
)}$ instead of the less cumbersome
$\mathscr F L^q_\omega $, because in forthcoming papers
(cf. \cite{JPTT, PTT2}), we often assume that $\omega $ is of the
particular form $\omega (\xi )=\eabs \xi ^s$, and in this case we set
$\mathscr F L^q_s =\mathscr F L^q_{(\omega )}$ without brackets for
the weight parameter.

\par

\begin{rem} \label{x-independence}
In many situations it is convenient to permit an $x$
dependency for the weight $\omega$ in the definition of Fourier
Lebesgue spaces. More precisely, for each $\omega \in \mathscr P(\rr
{2d})$ we let $\mathscr FL^q_{(\omega )}$ be the set of all $f\in
\mathscr S'(\rr d)$ such that
$$
\nm f{\mathscr FL^q_{(\omega )}}
\equiv \nm {\widehat f\, \omega (x,\cdo )}{L^q}
$$
is finite. Since $\omega$ is
$v$-moderate for some $v\in \mathscr P(\rr {2d})$ it follows that
different choices of $x$ give rise to equivalent norms.
Therefore the condition
$\nm f{\mathscr FL^{q}_{(\omega )}}<\infty$ is
 independent of $x$, and  it follows that $\mathscr FL^q_{(\omega )}(\rr
d)$ is  independent of $x$ although $\nm \cdo {\mathscr
FL^{q}_{(\omega )}}$ might depend on $x$.
\end{rem}

\par

\subsection{Pseudodifferential operators} \label{subsec1}

In this subsection we recall some facts from Chapter XVIII in \cite {Ho1}
concerning pseudo-differential operators. Let $a\in
\mathscr S(\rr {2d})$, and let $t\in \mathbf R$ be fixed. Then
the pseudo-differential operator $\op _t(a)$ which corresponds to the
symbol $a$ is the linear and
continuous operator on $\mathscr S(\rr d)$, defined by
\begin{equation}\label{e0.5}
(\op _t(a)f)(x)
\\
=
(2\pi ) ^{-d}\iint a((1-t)x+ty,\xi )f(y)e^{i\scal {x-y}\xi }\,
dyd\xi .
\end{equation}
If  $t=0$, then $\op _t(a)$ is the Kohn-Nirenberg
representation $a(x,D)=\op (a)=\op _0(a)$, and if $t=1/2$, then $\op
_t(a)$ is the Weyl quantization of $a$.

\par

For general $a\in \mathscr S'(\rr {2d})$, the pseudo-differential
operator $\op _t(a)$ is defined as the operator with distribution
kernel
\begin{equation}\label{weylkernel}
K_{t,a}(x,y)=(2\pi )^{-d/2}(\mathscr F_2^{-1}a)((1-t)x+ty,x-y).
\end{equation}
Here $\mathscr F_2F$ is the partial Fourier transform of
$F(x,y)\in \mathscr S'(\rr{2d})$ with respect to the $y$-variable.
We remark that $K_{t,a}$ makes sense as a distribution in $\mathscr
S'(\rr {2d})$. In fact, the map
$$
F(x,y)\mapsto F((1-t)x+ty,x-y)
$$
is obviously a homeomorphism on $\mathscr S(\rr {2d})$  and on
$\mathscr S'(\rr {2d})$. Furthermore,
by straight-forward computations it follows that the partial Fourier
transform $\mathscr F_2$ is a homeomorphism on $\mathscr S(\rr {2d})$,
and extends in the usual way to a homeomorphism on $\mathscr S'(\rr
{2d})$ which is unitary on $L^2$ (cf. e.{\,}g. Section 7.1 in
\cite{Ho1}). Consequently, if $a\in \mathscr
S'(\rr {2d})$, then $K_{t,a}$ in
\eqref{weylkernel} makes sense as a tempered distribution, and $\op
_t(a)$ is a continuous operator from
$\mathscr S (\rr d)$ to $\mathscr S'(\rr d)$. As a consequence of
Schwartz kernel theorem it follows that the map $a\mapsto \op _t(a)$
is bijective from $\mathscr S'(\rr {2d})$ to the set of linear and
continuous operators from $\mathscr S (\rr d)$ to $\mathscr S'(\rr
d)$.

\par

We also note that $K_{t,a}$ belongs to $\mathscr S(\rr {2d})$, if and
only if $a\in \mathscr S(\rr {2d})$, and that the latter definition of
$\op _t(a)$ agrees with the operator in \eqref{e0.5} when $a\in
\mathscr S(\rr {2d})$.

\par

If $a\in \mathscr S'(\rr {2d})$ and $s,t\in
\mathbf R$, then there is a unique $b\in \mathscr S'(\rr {2d})$ such
that $\op _s(a)=\op _t(b)$. By straight-forward applications of
Fourier's inversion  formula, it follows that
\begin{equation}\label{pseudorelation}
\op _s(a)=\op _t(b) \quad \Longleftrightarrow \quad b(x,\xi
)=e^{i(t-s)\scal
{D_x}{D_\xi}}a(x,\xi ).
\end{equation}
(Cf. Section 18.5 in \cite{Ho1}.)

\par

Next we discuss symbol classes which will be used in the sequel.
Let $ \rho, \delta \in \mathbf R$ and let $\omega _0 \in \mathscr P_{\rho
,\delta } (\rr {2d})$. Then the symbol class
$S_{\rho ,\delta}^{(\omega _0 )}(\rr {2d})$ consists of all $a\in
C^\infty (\rr {2d})$ such that
\begin{equation}\label{Somegadef}
|\partial _x^\alpha \partial _\xi ^\beta a(x,\xi )|\le
C_{\alpha ,\beta }\omega _0 (x,\xi )\eabs \xi ^{-\rho |\beta
|+\delta |\alpha |}.
\end{equation}
It is clear that $S_{\rho ,\delta}^{(\omega _0 )}$ is a Frech\'et space
with semi-norms given by the smallest constant which can be used in
\eqref{Somegadef}.

\par

If $\omega _0 (x,\xi )=\eabs \xi^r$, then
$S^{(\omega _0 )}_{\rho ,\delta} (\rr {2d})$ agrees with the
H{\"o}rmander class $S^r_{\rho ,\delta}(\rr {2d})$.
Usually it is assumed that $0\le \delta < \rho \le 1$.

\par

The following result shows that pseudo-differential operators with
symbols in $S^{(\omega _0 )}_{\rho ,\delta}$ behave well.

\par

\begin{prop}
Let $a\in S^{(\omega _0 )}_{\rho ,\delta} (\rr {2d})$
where $\omega _0 \in \mathscr P_{\rho ,\delta}
(\rr {2d})$,
 $\rho ,\delta \in [0,1]$, $0\le \delta \le
\rho \le 1$ and $\delta <1$. Then $\op _t(a)$ is continuous on
$\mathscr S(\rr d)$ and extends uniquely to a continuous operator on
$\mathscr S'(\rr d)$.
\end{prop}

\par

\begin{proof}
We have $S_{\rho ,\delta}^{(\omega _0 )}(\rr {2d})=S(\omega _0
,g_{\rho,\delta})$, when $g=g_{\rho ,\delta}$ is the Riemannian metric
on $\rr {2d}$, defined by the formula
$$
\big (g_{\rho ,\delta }\big )_{(y,\eta )}(x,\xi ) = \eabs \eta
^{2\delta}|x|^2 + \eabs \eta ^{-2\rho}|\xi |^2
$$
(cf. Section 18.4--18.6 in \cite{Ho1}).
>From the assumptions it follows that $g_{\rho ,\delta}$ is slowly
varying, $\sigma$-temperate and satisfies $g_{\rho ,\delta}\le
g_{\rho ,\delta}^\sigma$, and that $\omega _0$ is $g_{\rho
,\delta}$-continuous and $\sigma$, $g_{\rho ,\delta}$-temperate (cf.
e.{\,}g. Sections 18.4--18.6 in \cite{Ho1} for definitions). The
result is now a consequence of Proposition 18.5.10 and Theorem
18.6.2 in \cite{Ho1}. The proof is complete.
\end{proof}

\par

\subsection{Sets of characteristic points} \label{subsec2}

\par

In this subsection we define the set of characteristic points of a symbol
$a\in S^{(\omega _0 )}_{\rho ,\delta} (\rr {2d})$, when $\omega _0 \in
\mathscr P_{\rho ,\delta}(\rr {2d})$ and $0\le \delta < \rho \le
1$. As remarked in the introduction, this definition is slightly
different comparing to \cite[Definition 18.1.5]{Ho1} in view of Remark
\ref{compchar} below.

\par

Let $R>0$, $X\subseteq \rr d$ be open and let $\Gamma \subseteq \rr
d\setminus 0$ be an open cone. For convenience we say that an element $c\in
S^0_{\rho ,\delta}(\rr {2d})$ is \emph{$(X,\Gamma ,R)$-unitary} when
$$
c(x,\xi )=1\quad \text{when}\quad x\in X,\ \xi \in \Gamma\
\text{and}\ |\xi |\ge R.
$$
If $(x_0,\xi _0)\in \rr d\times (\rr
d\setminus 0)$, then the element $c$ is called \emph{unitary near
$(x_0,\xi _0)$} if $c$ is $(X,\Gamma ,R)$-unitary for some open
neighbourhood $X$ of $x_0$, open conical neighbourhood $\Gamma$ of
$\xi _0$ and $R>0$.

\par

\begin{defn}\label{defchar}
Let  $0\le \delta <\rho \le 1$, $\omega _0 \in \mathscr P _{\rho,\delta}
(\rr {2d})$, $a\in S^{(\omega _0 )}_{\rho
,\delta} (\rr {2d})$, and set $\mu =\rho -\delta >0$.
The point $(x_0,\xi _0)$ in $\rr d\times (\rr d\setminus 0)$ is
called \emph{non-characteristic} for $a$ (with respect to $\omega _0$),
if there are elements $b\in S^{(1/\omega _0
)}_{\rho ,\delta} (\rr {2d})$, $c\in S^0_{\rho ,\delta}(\rr
{2d})$ which is unitary near $(x_0,\xi _0)$, and $h\in S^{-\mu}_{\rho
,\delta}(\rr {2d})$ such that
$$
b(x,\xi )a(x,\xi ) = c(x,\xi ) + h(x,\xi ),\qquad (x,\xi )\in \rr
{2d}.
$$
The point $(x_0,\xi _0)$ in $\rr d\times (\rr d\setminus 0)$ is called
\emph{characteristic} for $a$ (with respect to $\omega _0$), if it is
\emph{not} non-characteristic for $a$ with respect to $\omega _0 $. The
set of characteristic points (the characteristic set), for
$a\in S_{\rho ,\delta}^{(\omega _0 )}(\rr {2d})$ with respect to $
\omega _0$, is denoted by $\Char (a)=\Char _{(\omega _0)}(a)$.
\end{defn}

\par

\begin{rem}\label{compchar}
Let $\omega _0 (x,\xi )=\eabs \xi ^r$, $r\in \mathbf R$, and assume that
$a\in S^r_{1,0}(\rr {2d})=S_{1,0}^{(\omega _0 )}(\rr {2d})$ is
polyhomogeneous with principal symbol $a_r\in S^r_{1,0}(\rr
{2d})$. (Cf. Definition 18.1.5 in \cite{Ho1}.) Also let $\Char '(a)$
be the set of characteristic points of $\op (a)$ in the classical
sense (i.{\,}e. in the sense of Definition 18.1.25 in \cite{Ho1}). We
claim that
\begin{equation}\label{charcharrel}
\Char _{(\omega _0 )}(a)\subseteq \Char '(a).
\end{equation}

\par

In fact, assume that $(x_0,\xi _0)\notin \Char '(a)$. This means that
there is a neighbourhood $X$ of $x_0$, a conical neighbourhood $\Gamma
$ of $\xi _0$, $R>0$ and $b\in S^{-r}_{1,0}(\rr {2d})$ such that
$a_r(x,\xi )b(x,\xi )=1$
when
\begin{equation}\label{UGammaR}
x\in X,\quad \xi \in \Gamma,\quad \text{and}\quad |\xi |>R.
\end{equation}
We shall prove that $a(x,\xi )b(x,\xi )=1$ when $(x,\xi )$
satisfies \eqref{UGammaR} for some $b\in S^{-r}_{1,0}$ and some
choices of $X$, $\Gamma$ and $R$, wherefrom $(x_0,\xi
_0)\notin \Char _{(\omega _0 )}(a)$.

\par

Since
\begin{gather*}
|b(x,\xi )|\le C\eabs \xi ^{-r},\quad |a_r(x,\xi )|\le C\eabs \xi
^{r},
\intertext{and}
|a(x,\xi )-a_r(x,\xi )|\le C\eabs \xi ^{r-1},
\end{gather*}
for some constant $C$, it follows that
$$
|a_r(x,\xi )|\ge C^{-1}\eabs \xi ^r,\quad \text{and}\quad C^{-1}\eabs
\xi ^r\le |a(x,\xi )|\le C\eabs \xi ^r,
$$
when $(x,\xi )$ satisfies \eqref{UGammaR}, for some constants $C$ and
$R$. Hence, if $\fy \in S^0_{1,0}$ is supported by $X\times \Gamma $,
and equal to one in a conical neighborhood of $(x_0,\xi _0)$, it
follows that $b=\fy \cdot a$ fulfills the required properties. This
proves the assertion.
\end{rem}

\par

\subsection{Modulation spaces} \label{subsec3}

\par

In this subsection we consider properties of
modulation spaces which will be used
for the definition of wave-front sets of such spaces in Sections \ref{sec6},
and for the proofs of
micro-local results for pseudo-differential operators with
non-smooth symbols in Section \ref{sec7}. The reader who is not
interested in these investigations might immediately pass to Sections
\ref{sec2} and \ref{sec3}.

\par

The short-time Fourier transform of $f\in \mathscr S'(\rr d)$
with respect to (the fixed window function) $\phi \in \mathscr S(\rr
d)$ is defined by
$$
(V_\phi f)(x,\xi ) =\mathscr F(f\cdot \overline {\phi (\cdo -x)})(\xi
).
$$
We note that the right-hand side makes sense, since it is the partial
Fourier transform of the tempered distribution
$$
F(x,y)=(f\otimes \overline \phi)(y,y-x)
$$
with respect to the $y$-variable.
If $f \in L^p_{(\omega )}(\rr d)$ for some $p\in [1,\infty ]$
and $\omega \in \mathscr P(\rr d)$, then  $V_\phi f$
takes the form
\begin{equation}\label{stftformula}
V_\phi f(x,\xi ) = (2\pi )^{-d/2}\int _{\rr {d}} f(y)\overline {\phi
(y-x)}e^{-i\scal y\xi}\, dy.
\end{equation}

\par

In the following lemma we recall some
general continuity properties of the short-time Fourier transform. We
omit the proof since the result can be found in
e.{\,}g. \cite{Fo,Gro-book}.

\par

\begin{lemma}\label{stftlemma}
Let $f\in \mathscr S'(\rr d)$ and $\phi \in \mathscr S(\rr d)$.
Then the following is true:
\begin{enumerate}
\item $V_\phi f\in
\mathscr S(\rr {2d})$ if and only if $f,\phi \in \mathscr S(\rr {d})$;

\vrum

\item The map $(f,\phi )\mapsto V_\phi f$ is continuous from $\mathscr
S(\rr d)\times \mathscr S(\rr d)$ to $\mathscr S(\rr {2d})$, which
extends uniquely to a continuous map from
$L^2(\rr d)\times L^2(\rr d)$ to $L^2(\rr {2d})$.
\end{enumerate}
\end{lemma}

\par

If $f\in \mathscr S'(\rr d)$ and $\phi \in \mathscr S(\rr d)$, then it
is well-known that  $V_\phi f\in \mathscr S'\cap
C^\infty$ and
$$
|V_\phi f(x,\xi )|\le C\eabs x^{N_0}\eabs \xi ^{N_0},
$$
for some constants $C$ and $N_0$ (see e.{\,}g. \cite{Gro-book}). If in
addition $f$ has compact support, then the following estimate holds
(see Proposition 3.6 in \cite{CG}).

\par

\begin{lemma}\label{stftcompact}
Let $f\in \mathscr E'(\rr d)$ and $\phi \in \mathscr S(\rr d)$. Then
for some constant $N_0$ and every $N\ge 0$, there are constants $C_N$
such that
$$
|V_\phi f(x,\xi )|\le C_N\eabs x^{-N}\eabs \xi ^{N_0}.
$$
\end{lemma}

\par

For further investigations of the short-time Fourier transform, we
need the twisted convolution $\widehat *$ on $L^1(\rr {2d})$,
defined by the formula
$$
(F\, \widehat *\, G)(x,\xi )=(2\pi )^{-d/2}\iint F(x-y,\xi -\eta
)G(y,\eta )e^{-i\scal {x-y}\eta}\, dyd\eta .
$$
By straight-forward computations it follows that $\widehat *$
restricts to a continuous multiplication on $\mathscr S(\rr
{2d})$. Furthermore, the map $(F,G)\mapsto F\, \widehat *\, G$ from
$\mathscr S(\rr {2d})\times \mathscr S(\rr {2d})$ to $\mathscr S(\rr
{2d})$ extends uniquely to continuous mappings from $\mathscr S'(\rr
{2d})\times \mathscr S(\rr {2d})$ and $\mathscr S(\rr {2d})\times
\mathscr S'(\rr {2d})$ to $\mathscr S'(\rr {2d})\bigcap C^{\infty}(\rr
{2d})$.

\par

The following lemma is important when proving certain invariance
properties for modulation spaces. We omit the proof since the
result can be found in \cite{Gro-book}.

\par

\begin{lemma}\label{stftproperties}
For each $f\in \mathscr S'(\rr d)$ and  $\phi _j\in
\mathscr S(\rr d)$,  $j=1,2,3$, it holds
$$
(V_{\phi _1}f)\widehat *(V_{\phi _2}\phi _3) = (\phi _3,\phi
_1)_{L^2}\cdot V_{\phi _2}f.
$$
\end{lemma}

\medspace

Now we recall the definition of modulation spaces. Let $\omega
\in \mathscr P(\rr {2d})$, $p,q\in
[1,\infty ]$, and  $\phi \in \mathscr S(\rr d)\setminus 0$ be
fixed. Then the \emph{modulation space} $M^{p,q}_{(\omega )}(\rr
d)$ is the set of all $f\in \mathscr S'(\rr d)$ such that
\begin{equation}\label{modnorm}
\nm f{M^{p,q}_{(\omega )}} = \nm f{M^{p,q,\phi }_{(\omega )}} \equiv
\nm {V_\phi f\, \omega}{L^{p,q}_1}<\infty .
\end{equation}
Here $\nm \cdo {L^{p,q}_1}$ is the norm given by
$$
\nm F{L^{p,q}_1}\equiv \Big ( \int _{\rr {d}} \Big ( \int _{\rr
{d}}|F(x,\xi )|^p\, dx \Big )^{q/p} \, d\xi \Big
)^{1/q},
$$
when $F\in L^1_{loc}(\rr {2d})$ (with obvious interpretation when
$p=\infty$ or $q=\infty$). Furthermore, the modulation space
$W^{p,q}_{(\omega )}(\rr d)$ consists of all $f\in \mathscr S'(\rr d)$
such that
\begin{equation*}
\nm f{W^{p,q}_{(\omega )}} = \nm f{W^{p,q,\phi }_{(\omega )}} \equiv
\nm {V_\phi f\, \omega}{L^{p,q}_2}<\infty ,
\end{equation*}
where $\nm \cdo {L^{p,q}_2}$ is the norm given by
$$
\nm F{L^{p,q}_2}\equiv \Big ( \int _{\rr {d}} \Big ( \int _{\rr
{d}}|F(x,\xi )|^q\, d\xi \Big )^{p/q} \, dx \Big )^{1/p},
$$
when $F\in L^1_{loc}(\rr {2d})$.

\par

If $\omega =1$, then the notation $M^{p,q}$ and $W^{p,q}$
are used instead of $M^{p,q}_{(\omega )}$ and $W^{p,q}_{(\omega )}$
respectively. Moreover we set $M^p_{(\omega )}=W^p_{(\omega )} =
M^{p,p}_{(\omega )}$ and $M^p=W^p = M^{p,p}$.

\par

We note that $M^{p,q}$ are modulation spaces of classical form, while
$W^{p,q}$ are classical form of Wiener amalgam spaces. We refer to
\cite{Feichtinger6} for the most updated definition of modulation
spaces.

\par

The following proposition is a consequence of well-known facts
in \cite {F1} or \cite {Gro-book}. Here and in what follows we let
$p'$ denote the conjugate exponent of $p$, i.{\,}e. $1/p+1/p'=1$.

\par

\begin{prop}\label{p1.4}
Let $p,q,p_j,q_j\in [1,\infty ]$ for $j=1,2$, and let
$\omega ,\omega _1,\omega _2,v\in \mathscr P(\rr {2d})$. Then the
following is true:
\begin{enumerate}
\item[{\rm{(1)}}] if  $\omega$ is $v$-moderate and if $\phi \in
M^1_{(v)}(\rr d)\setminus 0$, then $f\in
M^{p,q}_{(\omega )}(\rr d)$ if and only if \eqref {modnorm} holds,
i.{\,}e. $M^{p,q}_{(\omega )}(\rr d)$ is independent of the choice of
$\phi$. Moreover, $M^{p,q}_{(\omega )}(\rr d)$ is a Banach space under
the norm in \eqref{modnorm}, and different choices of $\phi$ give rise
to equivalent norms;

\vrum

\item[{\rm{(2)}}] if $p_1\le p_2$, $q_1\le q_2$ and $\omega _2\le C
\omega _1$ for some constant $C$, then
$$
\mathscr S(\rr d)\hookrightarrow M^{p_1,q_1}_{(\omega _1)}(\rr
d)\hookrightarrow M^{p_2,q_2}_{(\omega _2)}(\rr d)\hookrightarrow
\mathscr S'(\rr d)\text ;
$$

\vrum

\item[{\rm{(3)}}] the sesqui-linear form $( \cdo ,\cdo )$ on $\mathscr
S$ extends
to a continuous map from $M^{p,q}_{(\omega )}(\rr
d)\times M^{p'\! ,q'}_{(1/\omega )}(\rr d)$ to $\mathbf C$. On the
other hand, if $\nmm a = \sup \abp {(a,b)}$, where the supremum is
taken over all $b\in M^{p',q'}_{(1/\omega )}(\rr d)$ such that
$\nm b{M^{p',q'}_{(1/\omega )}}\le 1$, then $\nmm {\cdot}$ and $\nm
\cdot {M^{p,q}_{(\omega )}}$ are equivalent norms;

\vrum

\item[{\rm{(4)}}] if $p,q<\infty$, then $\mathscr S(\rr d)$ is dense
in $M^{p,q}_{(\omega )}(\rr d)$. The dual space of $M^{p,q}_{(\omega
)}(\rr d)$ can be identified with $M^{p'\! ,q'}_{(1/\omega )}(\rr
d)$, through the form $(\cdo  ,\cdo )_{L^2}$. Moreover, $\mathscr
S(\rr d)$ is weakly dense in $M^{\infty }_{(\omega )}(\rr d)$.
\end{enumerate}
Similar facts hold when the $M^{p,q}_{(\omega)}$ spaces are replaced
by $W^{p,q}_{(\omega)}$ spaces.
\end{prop}

\par

Proposition \ref{p1.4}{\,}(1) permits us to be rather vague about
to the choice of $\phi \in  M^1_{(v)}\setminus 0$ in
\eqref{modnorm}. For example, $\nm a{M^{p,q}_{(\omega )}}\le C$  for
some $C>0$  and  for
every $a$ which belongs to a given subset of $\mathscr S'$,
means that the inequality holds for some choice
of $\phi \in  M^1_{(v)}\setminus 0$. Evidently, for any other choice
of $\phi \in  M^1_{(v)}\setminus 0$, a similar inequality is true
although $C$ may have to be replaced by a larger constant, if
necessary.

\par

Locally, the spaces $\mathscr FL^q_{(\omega)}(\rr d)$,
$M^{p,q}_{(\omega )}(\rr d)$ and $W^{p,q}_{(\omega )}(\rr d)$ are the
same, in the sense that
$$
\mathscr FL^q_{(\omega )}(\rr d) \ttbigcap \mathscr E '(\rr d) =
M^{p,q}_{(\omega )}(\rr d) \ttbigcap \mathscr E'(\rr d) =
W^{p,q}_{(\omega )}(\rr d) \ttbigcap \mathscr E'(\rr d),
$$
in view of Remark 4.4 in \cite{RSTT}. In Section \ref{sec2} and
\ref{sec6} we extend these properties in context of the new type of
wave-front sets, and recover the above equalities at the end of Section
\ref{sec5}.

\par

\begin{rem}\label{coorb}
An example of a space which might be considered as a modulation space
and which is neither  of the form
$M^{p,q}_{(\omega )}$ nor $W^{p,q}_{(\omega )}$ is the space $\widetilde
M_{(\omega )}$, which consists of all $a\in \mathscr S'(\rr {2d})$
such that
\begin{equation}\label{mtildenorm}
\nm a{\widetilde M_{(\omega )}} \equiv \int _{\rr {d}}\sup _{\zeta \in
\rr d} \big( \sup _{x,\xi  \in \rr d}|V_\phi
a(x,\xi ,\zeta ,z)\, \omega (x,\xi ,\zeta, z)| \big )\, dz
\end{equation}
is finite.  By straight-forward
application of Lemma \ref{stftproperties} and Young's inequality it
follows that
$$
M^{\infty ,1}_{(\omega )}(\rr {2d})\subseteq \widetilde M_{(\omega
)}(\rr {2d})\subseteq M^{\infty ,\infty}_{(\omega )}(\rr {2d}),
$$
with continuous embeddings. (Cf. e.{\,}g. \cite{Feichtinger3,
Feichtinger4, Feichtinger6, Gro-book}.)
\end{rem}

\par

\subsection{Pseudo-differential operators with non-smooth symbols}
\label{subsec4}

\par

In this subsection we discuss properties of pseudo-differential
operators in context of modulation spaces, and start with the following
special case of Theorem 4.2 in \cite{Toft4}. We omit the proof.

\par

\begin{prop}\label{p5.4}
Let $p,q,p_j,q_j\in [1,\infty ]$ for
$j=1,2$, be such that
$$
1/p_1-1/p_2=1/q_1-1/q_2=1-1/p-1/q,\quad q\le
p_2,q_2\le p .
$$
Also let $\omega \in \mathscr
P(\rr {2d}\oplus \rr {2d})$ and $\omega _1,\omega _2\in \mathscr P(\rr
{2d})$ be such that
\begin{equation}\label{e5.9}
\omega (x,\xi ,\zeta ,z ) \le C\frac {\omega _1
(x+z,\xi  )}{\omega _2(x,\xi +\zeta )}
\end{equation}
for some constant $C$. If $a\in M^{p,q}_{(1/\omega )}(\rr
{2d})$, then $\op (a)$ from $\mathscr S(\rr d)$ to $\mathscr S'(\rr
d)$ extends uniquely to a continuous mapping from
$M^{p_1,q_1}_{(\omega _1)}(\rr d)$ to $M^{p_2,q_2}_{(\omega _2)}(\rr
d)$.
\end{prop}

\par

In Section \ref{sec7} we discuss wave-front set properties for
pseudo-differential operators, where the symbol classes are defined by
means of modulation spaces. Here, for $\rho \in \mathbf R$ and $s\in
\rr 4$ we define
\begin{equation}\label{omegasrho}
\omega _{s,\rho}(x,\xi ,\zeta ,z) = \omega (x,\xi ,\zeta
,z)\eabs x^{-s_4}\eabs \zeta ^{-s_3}\eabs \xi ^{-\rho s_2}\eabs
z^{-s_1}.
\end{equation}

\par

\begin{defn}\label{symbolclass}
Let $\omega \in \mathscr P(\rr {2d}\oplus \rr {2d})$, $s\in \rr 4$,
$\rho \in \mathbf
R$, and let $\omega_{s,\rho}$ be  as in \eqref{omegasrho}. Then the
symbol class $\mho _{(\omega )}^{s,\rho}(\rr {2d})$ is the set of all
$a\in \mathscr S'(\rr {2d})$ which satisfy
\begin{equation*}
\partial _\xi ^\alpha a\in M^{\infty ,1}_{(1/\omega
_{s(\alpha),\rho})}(\rr {2d}),\qquad s(\alpha)= (s_1,|\alpha
|,s_3,s_4),
\end{equation*}
for each multi-indices $\alpha$ such that $|\alpha |\le 2s_2$.
\end{defn}

\par

It follows from the following lemma that symbol classes of the form
$\mho _{(\omega )}^{s,\rho}(\rr {2d})$ are interesting also in the
classical theory.

\par

\begin{lemma}\label{connection}
Let $\rho \in [0,1]$, $\omega \in \mathscr P_0 (\rr {2d}\oplus \rr
{2d})$ and $\omega _0\in \mathscr P _{\rho ,0} (\rr {2d})$ be such
that
$$
\omega _0(x,\xi )=\omega (x,\xi ,0,0).
$$
Then the following conditions are equivalent:
\begin{enumerate}
\item $a\in S^{(\omega _0)}_{\rho ,0} (\rr {2d})$,
i.{\,}e. \eqref{Somegadef} holds for $\omega =\omega _0$;

\vrum

\item $\omega _0^{-1}a\in S^0_{\rho ,0}$;

\vrum

\item $\eabs x^{-s_4}a\in \bigcap _{s_1,s_2,s_3\ge 0}\mho
_{(\omega )}^{s,\rho}(\rr {2d})$.
\end{enumerate}
\end{lemma}

\par

For the proof of Lemma \ref{connection} we note that
\begin{equation}\label{moreidentities}
\underset {s\ge 0} \ttbigcap M^{\infty ,1}_{(v_{r,s})}(\rr {2d}) =
S^{r}_{0,0}(\rr {2d}),  \quad   v_{r ,s}(x,\xi ,\zeta ,z) =
\eabs \xi ^{-r} \eabs \zeta ^s\eabs z^s,
\end{equation}
which follows from Theorem 2.2 in \cite{To8} (see also Remark 2.18 in
\cite{HTW}).

\par

\begin{proof}
In order to prove the equivalence between (1) and (2) we
note that the condition $\omega _0\in \mathscr P_{\rho ,0}$ implies
that $\omega _0\in S^{(\omega _0)}_{\rho ,0} (\rr {2d})$ and $\omega
_0^{-1}\in S^{(1/\omega _0)}_{\rho ,0}(\rr {2d})$. Hence, if
$a\in S^{(\omega _0)}_{\rho ,0} (\rr {2d})$, then it follows by
straight-forward computations that
$$
\omega _0^{-1}a \in S^{(1/\omega _0)}_{\rho ,0}\cdot S^{(\omega
_0)}_{\rho ,0} = S^{(\omega _0 ^{-1}\omega _0)}_{\rho ,0} =
S^{(1)}_{\rho ,0}= S^0_{\rho ,0}
$$
(see also Lemma 18.4.3 in \cite{Ho1}). This proves that (1) implies
(2), and in the same way the opposite implication follows.

\par

Next we consider (3). We observe that for some positive constants $C$
and $N$ we have
$$
C^{-1}\omega _0(x,\xi )\eabs \zeta ^{-N}\eabs z^{-N} \le \omega
(x,\xi ,\zeta ,z) \le C\omega _0(x,\xi )\eabs \zeta ^{N}\eabs z^{N} ,
$$
which implies that
$$
\underset {s_1,s_2,s_3\ge 0} \ttbigcap \mho _{(\omega )}^{s,\rho}(\rr
{2d}) = \underset {s_1,s_2,s_3\ge 0} \ttbigcap \mho _{(\omega
_0)}^{s,\rho}(\rr {2d}).
$$
Since the map $a\mapsto \omega _0^{-1}a$ is a homeomorphism from
$M^{\infty ,1}_{(\omega _1/\omega _0)}$ to $M^{\infty ,1}_{(\omega
_1)}$ when $\omega _1\in \mathscr P$, by Theorem 2.2 in \cite{To8}, we
may assume that $\omega =\omega _0=1$. Furthermore we may assume that
$s_4=0$, since (3) is invariant under the choice of $s_4$. For such
choices of parameters, the asserted equivalences can be formulated as
$$
\underset {s_1,s_2,s_3\ge 0} \ttbigcap \mho _{(\omega )}^{s,\rho}(\rr
{2d}) = S^0_{\rho ,0}(\rr {2d}),\qquad \omega =1.
$$

\par

The result is now an immediate consequence of \eqref{moreidentities}
and the fact that $a\in S^0_{\rho ,0}(\rr {2d})$, if and only if
for each multi-indices $\alpha$ and $\beta$, it holds
$$
\partial _x^\alpha \partial _\xi ^\beta a\in
 S^{-\rho |\beta |}_{0,0}(\rr {2d}).
$$
The proof is complete.
\end{proof}

\par

\section{Wave front sets with respect to Fourier
Lebesgue spaces}\label{sec2}

\par

In this section we define wave-front sets with respect to Fourier
Lebesgue spaces, and show some basic properties.

\par

Assume that $\omega \in \mathscr P(\rr {2d})$, $\Gamma
\subseteq \rr d\setminus 0$ is an open cone and $q\in [1,\infty ]$
are fixed. For any $f\in \mathscr S'(\rr d)$, let
\begin{equation}
|f|_{\mathscr FL^{q,\Gamma}_{(\omega )}} = |f|_{\mathscr
FL^{q,\Gamma}_{(\omega ), x}}
\equiv
\Big ( \int _{\Gamma} |\widehat f(\xi )\omega (x,\xi )|^{q}\, d\xi
\Big )^{1/q}\label{skoff}
\end{equation}
(with obvious interpretation when $q=\infty$). We
note that $|\cdo |_{\mathscr FL^{q,\Gamma}_{(\omega ), x}}$ defines
a semi-norm on $\mathscr S'$ which might attain the value $+\infty$.
Since $\omega $ is $v$-moderate for some $v \in \mathscr P(\rr
{2d})$, it follows that different $x \in \rr d$ gives rise to
equivalent semi-norms $ |f|_{\mathscr FL^{q,\Gamma}_{(\omega ),
x}}$. Furthermore, if $\Gamma =\rr d\setminus 0$, $f\in
\mathscr FL^{q}_{(\omega )}(\rr d)$ and $q<\infty$, then
$|f|_{\mathscr FL^{q,\Gamma}_{(\omega ), x}}$ agrees with the Fourier
Lebesgue norm $\nm f{\mathscr FL^{q}_{(\omega ), x}}$ of $f$.

\par

For the sake of notational convenience we set
\begin{equation} \label{notconv}
\mathcal B=\mathscr FL^q_{(\omega )}=\mathscr FL^q_{(\omega )} (\rr
d), \;\;\;
\mbox{and}
\;\;\;
|\cdo |_{\mathcal B(\Gamma )}=|\cdo |_{\mathscr
FL^{q,\Gamma}_{(\omega ), x}}.
\end{equation}
We let $ \Theta _{\mathcal B}(f)=\Theta _{\mathscr FL^{q} _{(\omega
)}} (f)$ be the set of all $ \xi \in \rr d\setminus 0 $ such that
$|f|_{\mathcal B(\Gamma )} < \infty$, for some $
\Gamma = \Gamma_{\xi}$.  We also let $\Sigma_{\mathcal B} (f)$
be the complement of $ \Theta_{\mathcal
B} (f)$ in $\rr d\setminus 0 $. Then
$\Theta_{ \mathcal B} (f)$ and $\Sigma_{\mathcal
B} (f)$ are open respectively
closed subsets in $\rr d\setminus 0$, which are independent of
the choice of $ x \in \rr d$ in \eqref{skoff}.

\par

\begin{defn}\label{wave-frontdef}
Let  $q\in [1,\infty ]$, $\omega \in \mathscr P(\rr {2d})$, $\mathcal
B$ be as in \eqref{notconv}, and let
$X$ be an open subset of $\rr d$.
The wave-front set of
$f\in \mathscr D'(X)$,
$
\WF _{\mathcal B}(f)  \equiv  \WF _{\mathscr FL^q_{(\omega )}}(f)
$
with respect to $\mathcal B$ consists of all pairs $(x_0,\xi_0)$ in
$X\times (\rr d \setminus 0)$ such that
$
\xi _0 \in  \Sigma _{\mathcal B} (\fy f)
$
holds for each $\fy \in C_0^\infty (X)$ such that $\fy (x_0)\neq
0$.
\end{defn}

\par

We note that $\WF  _{\mathcal B}(f)$ is a closed set in $\rr d\times
(\rr d\setminus 0)$, since it is obvious that its complement is
open. We also note that if $ x\in \rr d$ is fixed and $\omega _0(\xi
)=\omega (x,\xi )$, then $\WF _{\mathcal B} (f)=\WF _{\mathscr
FL^q_{(\omega _0)}}(f)$, since $\Sigma _{\mathcal B}$ is independent
of $x$.

\par

The following theorem shows that wave-front sets with respect to
$\mathscr FL^q_{(\omega )}$ satisfy appropriate micro-local
properties. It also shows that such wave-front sets are decreasing
with respect to the parameter $q$, and increasing with respect to the
weight $\omega$.

\par

\begin{thm}\label{wavefrontprop1}
Let $X \subseteq \rr d$ be open,
$q,r\in [1,\infty ]$ and $\omega ,\vartheta \in \mathscr P(\rr {2d})$
be such that
\begin{equation}\label{qandomega}
q\le r,\quad \text{and}\quad \vartheta (x,\xi )\le C\omega
(x,\xi ),
\end{equation}
for some constant $C$ which is independent of $x,\xi \in \rr d$. Also
let $\mathcal B$  be as in \eqref{notconv} and put
$ {\mathcal B _0} =\mathscr FL^r_{(\vartheta )}=\mathscr FL^r_{(\vartheta )}
(\rr d)$. If $f\in \mathscr D'(X)$ and $\fy \in C^\infty
(X)$, then $\WF _{ {\mathcal B _0} }(\fy \, f)\subseteq \WF _{\mathcal
B}(f)$.
\end{thm}

\par

\begin{proof}
It suffices to prove
\begin{equation}\label{chi-subsetAA}
\Sigma_{ {\mathcal B _0} } (\fy f) \subseteq
\Sigma_{\mathcal B} (f).
\end{equation}
when $ \fy \in  \mathscr S(\rr d)$ and $f\in
\mathscr E'(\rr d)$, since the statement only involve local
assertions. For the same reasons we may assume that $\omega (x,\xi )
=\omega (\xi )$ is independent of $x$. It is also no restrictions to
assume that $\vartheta =\omega$.

\par

Let $ \xi_0 \in \Theta_{\mathcal B} (f)$,
and choose open cones $\Gamma _1$ and $\Gamma_2$ in $\rr d$ such that
$\overline {\Gamma _2}\subseteq \Gamma _1$. Since $f$ has compact
support, it follows that $|\widehat f(\xi )\omega(\xi )|\le
C\eabs \xi ^{N_0}$ for some positive constants $C$ and $N_0$. The result
therefore follows if we prove that for each $N$, there are constants
$C_N$ such that
\begin{multline}\label{cuttoff1}
|\fy f|_{ {\mathcal B _0} (\Gamma _2)}\le C_N \Big (|
f|_{\mathcal B(\Gamma _1)} +\sup _{\xi \in \rr
d} \big ( |\widehat f(\xi )\omega(\xi )|\eabs \xi ^{-N} \big )
\Big )
\\[1ex]
\text{when}\quad \overline \Gamma _2\subseteq \Gamma
_1\quad \text{and}\quad N=1,2,\dots .
\end{multline}

\par

By using the fact that $\omega$ is $v$-moderate for some $v\in
\mathscr P(\rr d) $, and letting $F(\xi )=|\widehat f(\xi )\omega
(\xi ) |$ and $\psi (\xi )=|\widehat \fy (\xi )v (\xi )|$, it
follows that $\psi$ turns rapidly to zero at infinity and
\begin{multline*}
|\fy f| _{ {\mathcal B _0}(\Gamma _2)} = \Big (\int
_{\Gamma _2}|\mathscr F(\fy f)(\xi )\omega(\xi )|^{r}\, d\xi
\Big )^{1/r}
\\[1ex]
\le C\Big (\int _{\Gamma _2}\Big ( \int _{\rr {d}} \psi (\xi -\eta
)F(\eta )\, d\eta \Big )^{r}\, d\xi \Big )^{1/r} \le C(J_1+J_2)
\end{multline*}
for some constant $C$, where
\begin{align*}
J_1 &= \Big (\int _{\Gamma _2}\Big ( \int _{\Gamma _1}\psi (\xi
-\eta )F(\eta )\, d\eta \Big )^{r}\, d\xi \Big )^{1/r}
\intertext{and}
J_2 &= \Big (\int _{\Gamma _2}\Big ( \int _{\complement \Gamma _1}\psi
(\xi -\eta )F(\eta )\, d\eta \Big )^{r}\, d\xi \Big )^{1/r}.
\end{align*}

\par

Let $q_0$ be chosen such that $1/q_0+1/q=1+1/r$, and let $\chi
_{\Gamma _1}$ be the characteristic function of $\Gamma _1$. Then
Young's inequality gives
\begin{multline*}
J_1 \le  \Big (\int _{\rr {d}} \Big ( \int _{\Gamma _1}\psi (\xi
-\eta )F(\eta )\, d\eta \Big )^{r}\, d\xi \Big )^{1/r}
\\[1ex]
=\nm {\psi *(\chi _{\Gamma _1} F)}{L^{r}} \le \nm \psi {L^{q_0}}\nm
{\chi _{\Gamma _1}
F}{L^{q}} = C_\psi |f|_{\mathcal B(\Gamma _1)},
\end{multline*}
where $C_\psi = \nm \psi {L^{q_0}}<\infty$ since $\psi$ is turns
rapidly to zero at infinity.

\par

In order to estimate $J_2$, we note that the conditions $\xi \in
\Gamma _2$, $\eta \notin \Gamma _1$ and the fact that $\overline
{\Gamma _2}\subseteq \Gamma _1$ imply that  $|\xi -\eta |>c\max
(|\xi|,|\eta |)$ for some constant $c>0$, since this is true when
$1=|\xi |\ge |\eta|$. Since $\psi$ turns rapidly to zero at
infinity, it follows that for each $N_0> d$ and $N\in \mathbf N$ such
that $N > N_0$, we have
\begin{multline*}
J_2 \le  C_1\Big (\int _{\Gamma _2}\Big ( \int _{\complement
\Gamma _1}\eabs {\xi -\eta }^{-(2N_0+N)} F(\eta )\, d\eta \Big
)^{r}\, d\xi \Big )^{1/r}
\\[1ex]
\le C_2\Big (\int _{\Gamma _2}\Big ( \int _{\complement \Gamma
_1}\eabs {\xi}^{-N_0} \eabs \eta ^{-N_0}(\eabs \eta ^{-N} F(\eta
))\, d\eta \Big )^{r}\, d\xi \Big )^{1/r}
\\[1ex]
\le C_3\sup _{\eta \in \rr d}|\eabs \eta ^{-N} F(\eta ))|,
\end{multline*}
for some constants $C_1, C_2, C_3 > 0$.
This proves \eqref{cuttoff1}. The proof is complete.
\end{proof}

\par

\section{Wave-front sets for pseudo-differential
operators with smooth symbols}\label{sec3}

\par

In this section we establish mapping properties for
pseudo-differential operators on wave-front sets of Fourier Lebesgue
types. More precisely, we prove the following result (cf. (*) in the
abstract).

\par

\begin{thm}\label{mainthm2}
Let $\rho >0$, $\omega \in \mathscr P(\rr
{2d})$, $\omega _0 \in \mathscr P_{\rho ,0}(\rr
{2d})$, $a\in S^{(\omega _0)}_{\rho ,0} (\rr {2d})$, $f\in \mathscr
S'(\rr d)$ and $q\in [1,\infty ]$. Also let
$$
\mathcal B=\mathscr FL^q_{(\omega )}=\mathscr FL^q_{(\omega )} (\rr
d)\quad \text{and} \quad
\mathcal C=\mathscr FL^q_{(\omega /\omega _0)}=\mathscr FL^q_{(\omega
/\omega _0)}(\rr d).
$$
Then
\begin{equation}\label{wavefrontemb1}
\WF _{\mathcal C} (\op (a)f) \subseteq
\WF _{\mathcal B} (f)
\subseteq \WF _{\mathcal C} (\op (a)f)\ttbigcup
\Char _{(\omega )}(a).
\end{equation}
\end{thm}

\par

We need some preparations for the proof. The first proposition shows
that if $ x_0 \not\in \supp f $ then $ (x_0, \xi) \not\in
\WF _{\mathcal C} (\op (a)f) $ for every $\xi
\in \rr d \setminus 0$.

\par

\begin{prop}\label{propmain1AA}
Let $\omega \in \mathscr P(\rr {2d})$,  $\omega _0 \in \mathscr
P_{\rho ,\delta} (\rr {2d})$,
$0\le \delta \le \rho$, $0<\rho$, $\delta <1$, and let $a\in
S^{(\omega _0 )}_{\rho ,\delta}(\rr {2d})$. Also let $\mathcal C$ be
as in Theorem \ref{mainthm2} and let the operator
$L_a$ on $\mathscr S'(\rr d)$ be defined by the formula
\begin{equation}\label{Ladef}
(L_af)(x) \equiv  \fy _1(x)(\op (a)(\fy _2f))(x), \quad f\in \mathscr
S'(\rr d),
\end{equation}
where $\fy _1\in
C_0^{\infty}(\rr d)$ and $\fy _2 \in S_{0,0}^0(\rr d)$ are such that
$$
\supp \fy _1\bigcap \supp \fy _2=\emptyset .
$$
Then the kernel of $L_a$ belongs to
$\mathscr S(\rr {2d})$. In particular, the following is true:
\begin{enumerate}
\item $L_a =\op (a_0)$ for some $a_0\in \mathscr S(\rr {2d})$;

\vrum

\item $\WF _{\mathcal C}(L_af)=\emptyset$, for any given
$q\in [1,\infty ]$.
\end{enumerate}
\end{prop}

\par

\begin{proof}
We note that $a_0$ exists as a tempered distribution in view of
Section \ref{sec1}. We need to prove that $a_0\in \mathscr S$, or
equivalently, that the kernel $K_a$ of $L_a$ belongs to $\mathscr S$.
By the definition it follows that
$$
(L_af)(x) = (2\pi )^{-d}\iint a(x,\xi )\fy _1(x)\fy _2(y)f(y)e^{i\scal
{x-y}\xi}\, dyd\xi .
$$
Since $\fy _1$ has
compact support it follows that for some $\ep >0$ it holds $\fy
_1(x)\fy _2(y)=0$ when $|x-y|\le 2\ep$. Hence, if $\fy \in
C^{\infty}(\rr d)$ satisfy $\fy (x)=0$ when $|x|\le \ep$ and $\fy
(x)=1$ when $|x|\ge 2\ep$, $f_2=\fy _2f$ and $a_1=\fy _1a$, then it
follows by partial integrations that
\begin{multline*}
(L_af)(x) = \op (a_1)f_2(x) = (2\pi )^{-d}\iint a_1(x,\xi
)f_2(y)e^{i\scal {x-y}\xi}\, dyd\xi
\\[1ex]
= (2\pi )^{-d}\iint (-1)^{s_2}(\Delta _\xi ^{s_2}a_1)(x,\xi
)f_2(y)|x-y|^{-2s_2}e^{i\scal {x-y}\xi}\, dyd\xi
\\[1ex]
= (2\pi )^{-d}\iint (-1)^{s_2}(\Delta _\xi ^{s_2}a_1)(x,\xi
)f_2(y)\fy (x-y) |x-y|^{-2s_2}e^{i\scal {x-y}\xi}\, dyd\xi
\\[1ex]
=(\op (b_s)f)(x),
\end{multline*}
where $s_2=s\ge 0$ is an integer,
\begin{align}
b_s(x,y,\xi ) &= (-1)^{s_2}(\Delta _\xi ^{s_2}a_1)(x,\xi )\fy
_1(x)\fy _2(y)\fy (x-y)|x-y|^{-2s_2}\label{bsformula}
\intertext{and}
\op (b_s)f(x) &= (2\pi )^{-d}\iint b_s(x,y,\xi )f(y)e^{i\scal
{x-y}\xi}\, dyd\xi .\notag
\end{align}
>From the fact that $|x-y|\ge C\eabs {x-y}$, when $(x,y,\xi )\in
\supp b_s$, and that $a\in S^{(\omega _0)}_{\rho ,\delta}(\rr {2d})$,
it follows from \eqref{bsformula} that
\begin{multline*}
|b_s(x,y,\xi )|\le C_s\omega _0(x,\xi )\eabs x^{-2s}\eabs y^{-2s}\eabs
{\xi}^{-2\rho s}
\\[1ex]
\le C_s'\eabs x^{N_0-2s}\eabs y^{-2s}\eabs
{\xi}^{N_0-2\rho s},
\end{multline*}
for some constant $N_0$ which is independent of $s$. In the same way it
follows that
\begin{equation}\label{partderbs}
|\partial ^\alpha b_s(x,y,\xi )|\le C_{s,\alpha }\eabs x^{N_0-2s}\eabs
y^{-2s}\eabs {\xi}^{N_0-2\rho s},
\end{equation}
for some constant $N_0$ which depends on $\alpha$, but is independent
of $s$.

\par

Now let $N\ge 0$ be arbitrary. Since the distribution kernel $K_a$ of
$L_a$ is equal to
$$
(2\pi )^{-d}\int b_s(x,y,\xi )e^{i\scal {x-y}\xi}\, d\xi ,
$$
it follows by choosing $s$ large enough in \eqref{partderbs} that for
each multi-index $\alpha$, there is a constant $C_{\alpha ,N}$ such
that
$$
|\partial ^\alpha K_a(x,y)| \le C_{\alpha ,N} \eabs {x,y}^{-N}.
$$
This proves that $K_a\in \mathscr S(\rr {2d})$, and (1)
follows.

\par

The assertion (2) is an immediate consequence of (1). The proof is
complete.
\end{proof}

\par

Next we consider properties of the wave-front set of $\op (a)f$ at a
fixed point when $f$ is concentrated to that point.

\par

\begin{prop}\label{keyprop2AA}
Let $\rho$, $\omega$, $\omega _0$, $a$, $\mathcal B$ and $\mathcal
C$ be as in Theorem \ref{mainthm2}. Also let $q\in [1,\infty ]$ and
$f\in \mathscr E'(\rr d)$. Then the following is true:
\begin{enumerate}
\item if $\Gamma _1,\Gamma _2$ are open cones in $\subseteq \rr
d\setminus 0$ such that $\overline{\Gamma _2}\subseteq \Gamma _1$, and
$|f|_{\mathcal B(\Gamma _1)}<\infty$, then $|\op (a)f|_{\mathcal
C(\Gamma _2)}<\infty$;

\vrum

\item $\WF  _{\mathcal C}(\op (a)f)\subseteq
\WF _{\mathcal B}(f)$.
\end{enumerate}
\end{prop}

\par

We note that $\op (a)f$ in Proposition \ref{keyprop2AA} makes sense as
an element in $\mathscr S'(\rr d)$, by Proposition \ref{p5.4}.

\par

\begin{proof}
We may assume that $\omega (x,\xi )=\omega (\xi )$, $\omega
_0(x,\xi )=\omega _0(\xi )$, and that $\supp a\subseteq K\times \rr d$
for some compact set $K\subseteq \rr d$, since the
statements only involve local assertions. We only prove the result for
$q<\infty$. The slight modifications to the case $q=\infty$ are left
for the reader.

\par

By straight-forward computations we get
\begin{equation}\label{pseudoreform}
\mathscr F(\op (a)f)(\xi )  =(2\pi )^{-d/2} \int _{\rr {d}} (\mathscr
F_1a)(\xi -\eta,\eta )\widehat f(\eta )\, d\eta ,
\end{equation}
where $\mathscr F_1a$ denotes the partial Fourier transform of
$a(x,\xi )$ with respect to the $x$-variable. We need to estimate the
modulus of $\mathscr F_1a(\eta ,\xi )$.

\par

From the fact that $a\in S_{\rho ,0}^{(\omega _0)}(\rr {2d})$ is
smooth and compactly supported in the $x$ variable, it follows that
for each $N\ge 0$, there is a constant $C_N$ such that
\begin{equation}\label{F1aest}
|(\mathscr F_1a)(\xi ,\eta )|\le C_N\eabs \xi ^{-N}\omega _0(\eta ).
\end{equation}
Hence the facts that $\omega (\eta )\le \omega
(\xi )\eabs {\xi -\eta}^{N_0}$ and $\omega _0(\eta )\le \omega _0
(\xi )\eabs {\xi -\eta}^{N_0}$ for some $N_0$ give that for each
$N>d$ it holds
\begin{multline}\label{keyprop2AAest}
|(\mathscr F_1a)(\xi -\eta, \eta )\omega (\xi )/\omega _0(\xi )| \le
C_N\eabs {\xi -\eta} ^{-(N+2N_0)}\omega _0(\eta )\omega (\xi )/\omega
_0(\xi )
\\[1ex]
\le C_N'\eabs {\xi -\eta} ^{-N}\omega (\eta ),
\end{multline}
for some constants $C_N$ and $C_N'$.

\par

By letting $F(\xi )=|\widehat
f(\xi )\omega (\xi )|$, then \eqref{pseudoreform},
\eqref{keyprop2AAest} and H{\"o}lder's inequality give
\begin{multline}\label{estpseudo}
|\mathscr F(\op (a)f)(\xi )\omega _2(\xi )| \le C \int _{\rr {d}}
\eabs {\xi -\eta } ^{-N} F(\eta )\, d\eta ,
\\[1ex]
= C \int _{\rr {d}}\big ( \eabs {\xi -\eta} ^{-N/q} F(\eta
)\big ) \eabs {\xi -\eta} ^{-N/q'}\, d\eta
\\[1ex]
\le C'\Big ( \int _{\rr {d}} \eabs {\xi -\eta }^{-N} F(\eta
)^q\, d\eta \Big )^{1/q},
\end{multline}
where $C=(2\pi )^{-d/2}C_N''$ and
\begin{equation}\label{ineqineq2}
C'=C\nm {\eabs \cdot ^{-N}}{L^1}^{1/q'}.
\end{equation}
Here $C'<\infty$ in \eqref{ineqineq2}, since $N>d$.

\par

We have to estimate
$$
|(\op (a)f) |_{\mathcal C(\Gamma _2)} = \Big
(\int _{\Gamma _2}|\mathscr F(\op (a)f)(\xi )\omega (\xi )/\omega
_0(\xi )|^q\, d\xi \Big
)^{1/q}.
$$
By \eqref{estpseudo} we get
\begin{multline*}
\Big (\int _{\Gamma _2}|\mathscr F(\op (a)f)(\xi )\omega (\xi
)/\omega _0(\xi )|^q\, d\xi \Big )^{1/q}
\\[1ex]
 \le C \Big (\iint _{\xi \in \Gamma _2}\eabs
{\xi -\eta }^{-N} F(\eta )^q\, d\eta  d\xi \Big )^{1/q}
\le C(J_1+J_2),
\end{multline*}
for some constant $C$, where
\begin{align*}
J_1 &= \Big (\int _{\Gamma _2} \int _{\Gamma _1}\eabs {\xi
-\eta} ^{-N} F(\eta )^q\, d\eta  d\xi \Big )^{1/q}
\intertext{and}
J_2 &= \Big (\int _{\Gamma _2}  \int _{\complement \Gamma
_1}\eabs {\xi -\eta} ^{-N} F(\eta )^q\, d\eta  d\xi  \Big
)^{1/q}.
\end{align*}

\par

In order to estimate $J_1$ and $J_2$ we argue as in the proof of
\eqref{cuttoff1}. More precisely, for $J_1$ we have
\begin{align*}
J_1&\le \Big (\int _{\rr d} \int _{\Gamma _1}\eabs {\xi -\eta}
^{-N}F(\eta )^q\, d\eta  d\xi \Big )^{1/q}
\\[1ex]
&=  \Big (\int _{\rr d} \int _{\Gamma _1}\eabs {\xi} ^{-N} F(\eta
)^q\, d\eta  d\xi \Big )^{1/q}
= C\Big (\int _{\Gamma _1}F(\eta )^q\, d\eta \Big ) ^{1/q}<\infty .
\end{align*}

\par

In order to estimate $J_2$, we assume from now on that $\Gamma _2$
is chosen such that $\Gamma _2\subseteq \Gamma _1$, and that the
distance between the boundaries of $\Gamma _1$ and $\Gamma _2$ on
the $d-1$ dimensional unit sphere $\mathsf S^{d-1}$ is larger than
$r>0$. This gives
\begin{equation}\label{gammarel}
|\xi -\eta |> r\quad \text{when}\quad \xi \in \Gamma _2\ttbigcap
\mathsf S^{d-1},\ \text{and}\ \eta \in (\complement \Gamma
_1)\ttbigcap \mathsf S^{d-1}.
\end{equation}
Then for some constant $c>0$ we have
$$
|\xi -\eta |\ge c\max ( |\xi |,|\eta |),\quad \text{when}\quad \xi \in
\Gamma _2,\ \text{and}\ \eta \in \complement \Gamma _1.
$$
In fact, when proving this we may assume that
$|\eta |\le |\xi |=1$. Then we must have that $|\xi -\eta |\ge c$ for
some constant $c>0$, since we otherwise get a contradiction of
\eqref{gammarel}.

\par

Since $f$ has compact support, it follows that $F(\eta )\le C\eabs
\eta ^{t_1}$ for some constant $C$. By combining these estimates we
obtain
\begin{multline*}
J_2\le \Big (\int _{\Gamma _2}\int _{\complement \Gamma _1}F(\eta
)\eabs {\xi -\eta} ^{-N}\, d\eta  d\xi \Big )^{1/q}
\\[1ex]
\le C\Big ( \int _{\Gamma _2}\int _{\complement \Gamma _1}\eabs \eta
^{t_1}\eabs {\xi} ^{-N/2}\eabs \eta ^{-N /2}\, d\eta
d\xi \Big ).
\end{multline*}
Hence, if we choose $N>2d+2t_1$, it follows that the right-hand side
is finite. This proves (1).

\par

The assertion (2) follows immediately from (1) and the
definitions. The proof is complete.
\end{proof}

\par

We also need the following lemma. Here we recall Definition
\ref{defchar} for notations.

\par

\begin{lemma}\label{pseudolemma3}
Let $0\le \delta < \rho \le 1$, $\mu =\rho -\delta$,
$\omega _0\in \mathscr P_{\rho ,\delta}(\rr {2d})$, and
$a\in S^{(\omega _0)}_{\rho ,\delta} (\rr {2d})$. If  $(x_0,\xi
_0)\notin \Char _{(\omega _0)}(a)$, then for some open cone
$\Gamma =\Gamma _{\xi _0}$, open neighborhood $X\subseteq \rr d$ of
$x_0$ and $R>0$,
there are elements $c_j\in S^0_{\rho ,\delta}$ which are $(X,\Gamma
,R)$-unitary, $b_j\in S^{(1/\omega _0)}_{\rho ,\delta}$ and $h_j\in
S^{-j\mu}_{\rho ,\delta}$ for $ j\in \mathbf N$ such that
$$
\op (b_j)\op (a)=\op (c_j)+\op (h_j),\qquad j\ge 1.
$$
\end{lemma}

\par

\begin{proof}
For $j=1$, the result is obvious in view of Definition
\ref{defchar}. Therefore assume that $j>1$, and that $b_k$, $c_k$
and $h_k$ for $k=1,\dots ,j-1$ have already been chosen which
satisfy the required properties. Then we inductively define $b_j$ by
the formula
$$
\op (b_j) = (\op (c_{j-1})-\op (h_{j-1}))\op (b_{j-1}).
$$
By the inductive hypothesis it follows that
$$
\op (b_j)\op (a) =  \op (\widetilde c) +  \op (\widetilde h_1) +
\op (\widetilde h_2),
$$
where
\begin{align*}
\op (\widetilde c) &= \op (c_{j-1})\op (c_{j-1}),
\\[1ex]
\op (\widetilde h_1) &=[\op (c_{j-1}),\op (h_{j-1})]
\intertext{and}
\op (\widetilde h_2) &= \op (h_{j-1})\op (h_{j-1}).
\end{align*}
Here $[\cdo ,\cdo ]$ denotes the commutator between operators.

\par

By Theorems 18.5.4 and 18.5.10 in \cite {Ho1} it follows that
$\widetilde h_2\in
S^{-2(j-1)\mu}_{\rho ,\delta}\subseteq S^{-j\mu}_{\rho ,\delta}$,
since the conditions $j-1\ge 1$ and $0<\mu\le 1$ imply that
$-2(j-1)\mu\le -j\mu$. Hence
\begin{equation}\label{tildeh2}
\widetilde h_l\in S^{-j\mu}_{\rho ,\delta}
\end{equation}
holds for $l=2$.

\par

Next we consider the term $\op (\widetilde h_1)$. By Theorem 18.1.18
\cite{Ho1} it follows that
\begin{multline*}
\widetilde h_1 (x,\xi ) =i\sum _{|\alpha |=1} (\partial ^\alpha
_xc_{j-1}(x,\xi )\partial ^\alpha _\xi h_{j-1}(x,\xi ) - \partial
^\alpha _\xi c_{j-1}(x,\xi )\partial ^\alpha _x h_{j-1}(x,\xi ))
\\[1ex]
+\widetilde h_3(x,\xi ),
\end{multline*}
for some $\widetilde h_3\in S^{-(j+1)\mu}_{\rho ,\delta}$. Since the
sum belongs to $S^{-j\mu}_{\rho ,\delta}$ in view of the definitions,
it follows that \eqref{tildeh2} also holds for $l=1$.

\par

It remains to consider the term $\widetilde c$. By Theorems 18.5.4 and
18.5.10 in \cite {Ho1} again, it follows that
$$
\widetilde c(x,\xi ) = c_j(x,\xi ) +\widetilde h_4(x,\xi ),
$$
where
$$
c_{j}(x,\xi ) = \sum _{|\alpha |\le N} \frac{
(-i)^{|\alpha|}}{\alpha!}
\partial ^\alpha _xc_{j-1}(x,\xi )\partial ^\alpha _\xi
c_{j-1}(x,\xi ),
$$
and $\widetilde h_4\in S^{-j\mu}_{\rho ,\delta}$, provided $N$ was
chosen sufficiently large. Since $c_{j-1}=1$ on
$$
\sets {(x,\xi )}{x\in X,\ \xi \in \Gamma ,\ |\xi |>R},
$$
it follows that $c_j$ is $(X,\Gamma ,R)$-unitary. The
result now follows by letting $h_j=\widetilde h_1+\widetilde
h_2+\widetilde h_4$. The proof is complete.
\end{proof}

\par

\begin{proof}[Proof of Theorem \ref{mainthm2}]
We start to prove the first inclusion in \eqref{wavefrontemb1}.
Assume that $(x_0,\xi _0)\notin \WF _{\mathcal B}(f)$, let $\fy \in
C_0^\infty (\rr d)$ be such that $\fy =1$
in a neighborhood of $x_0$, and set $\fy _1=1-\fy$. Then it
follows from Proposition \ref{propmain1AA} that
$$
(x_0,\xi _0)\notin \WF _{\mathcal C}(\op
(a)(\fy _1f)).
$$
Furthermore, by Proposition \ref{keyprop2AA} we get
$$
(x_0,\xi _0)\notin \WF _{\mathcal C}(\op (a)(\fy f)),
$$
since if $a_0(x,\xi )=\fy (x)a(x,\xi )$, then $\op (a)(\fy f)$
is equal to $\op (a_0)(\fy f)$ near $x_0$. The first embedding in
\eqref{wavefrontemb1} is now a consequence of the inclusion
\begin{equation*}
\WF  _{\mathcal C}(\op (a)f)
\subseteq \WF _{\mathcal C}(\op (a)(\fy f))\ttbigcup \WF  _{\mathcal
C}(\op (a)(\fy _1f)).
\end{equation*}

\par

It remains to prove the last inclusion in \eqref{wavefrontemb1}. By
Proposition \ref{propmain1AA} it follows that it is no restriction to
assume that $f$ has compact support. Assume that
$$
(x_0,\xi _0)\notin  \WF _{\mathcal C}(\op
(a)f)\ttbigcup \Char _{(\omega _0)}(a),
$$
and choose $b_j$, $c_j$ and $h_j$ as in Lemma \ref{pseudolemma3}. We
shall prove that $(x_0,\xi _0)\notin  \WF  _{\mathcal B}(f)$. Since
$$
f = \op (1-c_j)f + \op (b_j)\op (a)f+\op (h_j)f,
$$
the result follows if we prove
$$
(x_0,\xi _0)\notin \mathsf S_1\ttbigcup \mathsf S_2\ttbigcup \mathsf
S_3,
$$
where
\begin{align*}
\mathsf S_1 &= \WF _{\mathcal B}(\op (1-c_j)f),\quad
\mathsf S_2 = \WF _{\mathcal B}(\op (b_j)\op (a)f)
\\[1ex]
\text{and}\quad \mathsf S_3 &= \WF _{\mathcal B}(\op
(h_j)f).
\end{align*}

\par

We start to consider $\mathsf S_2$. By the first embedding in
\eqref{wavefrontemb1} it follows that
$$
\mathsf S_2 = \WF _{\mathcal B}(\op (b_j)\op (a)f)
\subseteq \WF  _{\mathcal C}(\op (a)f).
$$
Since we have assumed that $(x_0,\xi _0)\notin \WF _{\mathcal C}(\op
(a)f)$, it follows that $(x_0,\xi _0)\notin \mathsf S_2$.

\par

Next we consider $\mathsf S_3$. Since $f$ has compact support and
$\omega ,\omega _0\in \mathscr P(\rr {2d})$, it follows from Lemma
\ref{pseudolemma3} that for
each $N\ge 0$, there is a $j\ge 1$ such that $\fy
(x)D^{\alpha}(\op (h_j)f)\in L^\infty $ when $|\alpha |\le N$ and
$\fy \in C_0^\infty (\rr d)$ for some $ h_j \in S ^{-j\rho}_{\rho,0}
$. This implies that $\mathsf S_3$ is empty, provided $N$ (and
therefore $j$) was chosen large enough.

\par

Finally we consider $\mathsf S_1$. By the assumptions it follows
that $a_0=1-c_j=0$ in $\Gamma$, and by replacing $ \Gamma $ with a
smaller cone, if necessary, we may assume that $ a_0 = 0 $ in a
conical neighborhood of $ \Gamma$. Hence, if $\Gamma
=\Gamma _1$, and $\Gamma _1$, $\Gamma _2$, $J_1$ and $J_2$ are the
same as in the
proof of Proposition \ref{keyprop2AA}, then it follows from that proof
and the fact that $a_0(x,\xi )\in S^0_{\rho ,0}$ is compactly
supported in the $x$-variable, that $J_1<+\infty$ and for each $N\ge
0$, there are constants $C_N$ and $C_N'$ such that
\begin{multline}\label{estagain4}
| \op (a_0)f |_{\mathcal B(\Gamma _2)} \le
C_N (J_1+J_2)
\\[1ex]
\le C_N' \Big (J_1+ \Big ( \int _{\Gamma _2}\int _{\complement
\Gamma _1}\eabs {\xi} ^{-N/2}\eabs \eta ^{(N_0-N/2)}\, d\eta d\xi
\Big )^{1/q}\Big ).
\end{multline}
for some $N_0\ge 0$. By choosing $N>2N_0+2d$, it follows that $|
\op (a_0)f |_{\mathcal C(\Gamma _2)}
< \infty$. This proves that $(x_0,\xi _0)\notin \mathsf S_1$, and
the proof is complete.
\end{proof}

\par

\begin{rem}
By Theorem 18.5.10 in \cite{Ho1} it follows that the first embedding
in \eqref{wavefrontemb1} remains valid if $\op (a)$ is replaced by
$\op _t(a)$.
\end{rem}

\par

\begin{rem}\label{rho0}
We note that the inclusions in Theorem \ref{mainthm2}
may be violated when $\omega _0=1$ and the assumption
$\rho >0$ is replaced by $\rho =0$. In fact, let $a(x,\xi )=e^{-i\scal
{x_0}\xi}$ for some fixed $x_0 \in \rr d$, and choose $\alpha$ in such
way that $f_\alpha (x) =\delta _0^{(\alpha )}$ does not belong to
$\mathcal B=\mathscr FL^q_{(\omega )}(\rr {d})$. Since
$$
(\op (a)f_\alpha )(x) = f_\alpha (x-x_0),
$$
by straight-forward computations, it follows that for some closed cone
$\Gamma$ in $\rr d\setminus 0$ we have
\begin{align*}
\WF _{\mathcal B} (f)  &= \sets {(0,\xi )}{\xi \in
\Gamma}
\\[1ex]
\WF _{\mathcal B} (\op (a)f)  &= \sets {(x_0,\xi )}{\xi
\in \Gamma},
\end{align*}
which are not overlapping when $x_0\neq 0$.
\end{rem}

\medspace

Next we apply Theorem \ref{mainthm2} on
operators which are elliptic with respect to $S^{(\omega _0)}_{\rho
,\delta}(\rr {2d})$, where $\omega _0\in \mathscr P(\rr {2d})$. More
precisely, assume that $0\le \delta <\rho \le1$ and $a\in S^{(\omega
_0)}_{\rho ,\delta}(\rr {2d})$. Then $a$ and
$\op (a)$ are called (locally) \emph{elliptic} with respect to
$S^{(\omega _0)}_{\rho ,\delta}(\rr {2d})$ or $\omega _0$,
if for each compact set $K\subseteq \rr d$, there are positive
constants $c$ and $R$ such that
$$
|a(x,\xi )| \ge c\omega _0(x,\xi ),\quad x\in K,\ |\xi |\ge R
$$
Since $|a(x,\xi )|\le C\omega _0(x,\xi )$, it follows from the
definitions that for each multi-index $\alpha$, there are constants
$C_\alpha$ such that
\begin{equation*}
|\partial ^\alpha _x\partial ^\beta _\xi a(x,\xi )| \le C_{\alpha
,\beta}|a(x,\xi )|\eabs \xi ^{-\rho |\beta |+\delta |\alpha |},\quad
x\in K,\   |\xi |>R.
\end{equation*}
(See e.{\,}g. \cite {Ho1,BBR}.) It immediately follows from the
definitions that $\Char _{(\omega _0)}(a)=\emptyset$ when $a$ is
elliptic with respect to $\omega _0$. The following result is now an
immediate consequence of Theorem \ref{mainthm2}.

\par

\begin{thm}\label{hypoellthm}
Let $\omega \in \mathscr P(\rr
{2d})$, $\omega _0\in \mathscr P_{\rho ,0}(\rr {2d})$, $q\in
[1,\infty ]$, $\rho >0$, and let $a\in S^{(\omega _0)} _{\rho ,0} (\rr
{2d})$ be elliptic with respect to $\omega _0$. Also let $\mathcal B$
and $\mathcal C$ be as in Theorem \ref{mainthm2}. If
$f\in \mathscr S'(\rr d)$, then
$$
\WF _{\mathcal C} (\op (a)f)= \WF _{\mathcal B} (f).
$$
\end{thm}

\par

\begin{cor}\label{cormainthm23B}
Assume that the hypothesis in Theorem \ref{hypoellthm} is
fulfilled with $\omega =\omega _0$, and let
$E$ be a parametrix for $\op (a)$. Then $\fy E\in \mathscr
FL^\infty _{(\omega )}$ for every $\fy \in C_0^\infty (\rr d)$,
i.{\,}e. for each $x_0\in $ and $\fy \in C_0^\infty (\rr d)$, there
is a constant $C$ such that \eqref{huvaligen} holds.
\end{cor}

\par

\begin{proof}
>From the assumptions it follows that $\op (a)E=\delta _0+\fy $ for
some $\fy \in C^\infty (\rr d)$. The result is then a consequence of
$$
\WF _{\mathscr FL^\infty _{(\omega _0)}} (E) = \WF _{\mathscr
FL^\infty } (\op (a)E) = \WF _{\mathscr FL^\infty } (\delta _0
+\fy )=\emptyset ,
$$
by Theorem \ref{hypoellthm}, where the last equality follows
from the fact that $\delta _0\in \mathscr FL^\infty$.
\end{proof}

\par

\begin{example}\label{examplheat}
Let $a(x,\xi )=a(\xi )$ be the symbol to the hypoelliptic partial
differential operator $\op (a)$ with constant coefficients
(cf. \cite[Chapter XI]{Ho1} for strict definitions). Then $a$ is
elliptic with respect to
$$
\omega (x,\xi )=\omega (\xi ) =(1+|a(\xi )|),
$$
which belongs to $\mathscr P_{\rho ,0}(\rr d)$ for some $\rho
>0$. Hence it follows from Theorem \ref{hypoellthm} and Corollary
\ref{cormainthm23B} that if $\fy \in C_0^\infty (\rr d)$, then
$$
|(1+|a(\xi )|)\mathscr F(\fy \cdot E)(\xi )|
\in L^\infty (\rr {d}).
$$

\par

An important hypoelliptic operator concerns the heat operator
$\partial _t-\Delta _x$, $(x,t)\in \rr {d+1}$, with symbol $a(x,t,\xi
,\tau )=|\xi |^2+i\tau$. In this case, $a$ is
elliptic with respect to
$$
\omega (x,t,\xi ,\tau )= (1+|\xi |^2+|\tau |).
$$
Hence it follows that
$$
(1+|\xi |^4+|\tau |^2)^{1/2}|\mathscr F(\fy \cdot E)(\xi ,\tau)|
\in L^\infty (\rr {d+1}),
$$
when $\fy \in C_0^\infty (\rr {d+1})$.

\par

For the heat operator we note that
$$
\Char '(a)=\sets {(x,t,0,\tau )}{x\in \rr d,\ t\in \mathbf R,\ \tau
\neq 0},
$$
which is not empty (see Remark \ref{compchar} for the definition of
$\Char '(a)$). Hence, $\Char _{(\omega _a)}(a)$ is strictly smaller
than $\Char '(a)$ in this case.
\end{example}

\par

\section{Wave-front sets of sup and inf types and
pseudo-differential operators}\label{sec5}

\par

In this section we put the micro-local analysis in a more general
context comparing to previous sections,
and define wave-front sets with respect to sequences of
Fourier Lebesgue type spaces. We also explain some consequences of
the investigations in previous sections in this general setting. For
example we show how one can obtain micro-local results which involve
only classical wave-front sets (cf. Remark \ref{remstandWF} and
Theorem \ref{thmclassicWFs} below).

\par

Let $\omega _j \in \mathscr
P(\rr {2d})$ and $q_j\in [1,\infty ]$ when $j$ belongs to some index
set $J$, and let $\mathcal B$ be the array of spaces, given by
\begin{equation}\label{notconvsequences}
 (\mathcal B_j) \equiv (\mathcal B_j)_{j\in J},\quad
\text{where}\quad \mathcal B_j=\mathscr FL^{q_j}_{(\omega
_j)}=\mathscr FL^{q_j}_{(\omega _j)} (\rr d), \quad j \in J.
\end{equation}

\par

If $f\in \mathscr S'(\rr d)$, and $(\mathcal B_j)$ is given by
\eqref{notconvsequences}, then we let
$\Theta_{(\mathcal B_j) }^{\sup}(f)$ be the set of all $\xi \in \rr
d\setminus 0$ such that for some $\Gamma = \Gamma _{\xi}$ and each
$j\in J$ it holds $|f|_{\mathcal B_j(\Gamma )} < \infty$. We
also let $\Theta _{(\mathcal B_j) }^{\inf}(f)$ be the set of all $
\xi \in \rr d\setminus 0 $ such that for some $\Gamma = \Gamma
_{\xi}$ and some $j\in J$ it holds $|f|_{\mathcal B_j(\Gamma)} <
\infty$. Finally we let $\Sigma _{(\mathcal B_j) }^{\sup} (f)$ and
$\Sigma _{(\mathcal B_j) }^{\inf} (f)$ be the complements in $\rr
d\setminus 0 $ of $\Theta_{(\mathcal B_j) }^{\sup}(f)$ and $\Theta
_{(\mathcal B_j) }^{\inf} (f)$ respectively.

\par

\begin{defn}\label{defsuperposWF}
Let $J$ be an index set, $q_j\in [1,\infty ]$, $\omega _j\in \mathscr
P(\rr {2d})$ when $j\in J$, $(\mathcal B_j)$ be as in
\eqref{notconvsequences}, and let $X$ be an open subset of $\rr d$.
\begin{enumerate}
\item The wave-front set of $f\in \mathscr D'(X)$,
of \emph{sup-type} with respect to $(\mathcal B_j) $,
$\WF ^{\, \sup} _{(\mathcal B_j) }(f) $,
 consists of all pairs
$(x_0,\xi_0)$ in $X\times (\rr d \setminus 0)$ such that
$
\xi _0 \in  \Sigma ^{\sup} _{(\mathcal B_j)} (\fy f)
$
holds for each $\fy \in C_0^\infty (X)$ such that $\fy (x_0)\neq
0$;

\vrum

\item The wave-front set of $f\in \mathscr D'(X)$,
of \emph{inf-type} with respect to $(\mathcal B_j)$,
$
\WF ^{\, \inf} _{(\mathcal B_j)}(f) $
consists of all pairs
$(x_0,\xi_0)$ in $X\times (\rr d \setminus 0)$ such that
$
\xi _0 \in  \Sigma ^{\sup} _{(\mathcal B_j)} (\fy f)
$
holds for each $\fy \in C_0^\infty (X)$ such that $\fy (x_0)\neq
0$.
\end{enumerate}
\end{defn}

\par

\begin{rem}\label{remstandWF}
Let $\omega _j(x,\xi ) = \eabs \xi ^{-j}$ for $j\in J=\mathbf
N_0$. Then it follows that $\WF _{(\mathcal B_j)}^{\, \sup}(f)$ in
Definition \ref{defsuperposWF} is equal to the standard wave front
set $\WF (f)$ in Chapter VIII in \cite{Ho1}.
\end{rem}

\par

The following result follows immediately from Theorems \ref{mainthm2}
and its proof. We omit the details. Here we let
\begin{equation} \label{notconvCsequence}
\mathcal C_j = \mathscr FL^{q_j}_{(\omega _j/\omega _0)}(\rr d)\quad
\text{and}\quad (\mathcal C_j) =(\mathcal C_j)_{j\in J}.
\end{equation}

\par

\renewcommand{\rubrik}{Theorem \ref{mainthm2}$'$}

\begin{tom}
Let $\rho >0$, $\omega _j\in \mathscr P(\rr
{2d})$, $\omega _0 \in \mathscr P_{\rho ,0}(\rr
{2d})$, $a\in S^{(\omega _0)}_{\rho ,0} (\rr {2d})$, $f\in \mathscr
S'(\rr d)$ and $q_j\in [1,\infty ]$ for $j\in J$. Also let
\begin{equation*}
(\mathcal B_j) \equiv (\mathcal B_j)_{j\in J}\quad \text{and}\quad
(\mathcal C_j) \equiv (\mathcal C_j)_{j\in J},
\end{equation*}
where
$$
\mathcal B_j = \mathscr FL^{q_j}_{(\omega _j)} = \mathscr
FL^{q_j}_{(\omega _j)}(\rr d) \quad \text{and}\quad  \mathcal C_j =
\mathscr FL^{q_j}_{(\omega _j/\omega _0)}=\mathscr FL^{q_j}_{(\omega
_j/\omega _0)}(\rr d).
$$
Then
\begin{multline}\tag*{(\ref{wavefrontemb1})$'$}
\WF ^{\, \sup}_{(\mathcal C_j)} (\op (a)f) \subseteq
\WF ^{\, \sup}_{(\mathcal B_j)} (f)
\\[1ex]
\subseteq \WF ^{\, \sup}_{(\mathcal C_j)} (\op (a)f)\ttbigcup
\Char _{(\omega _0)}(a),
\end{multline}
and
\begin{multline}\tag*{(\ref{wavefrontemb1})$''$}
\WF ^{\, \inf}_{(\mathcal C_j)} (\op (a)f) \subseteq
\WF ^{\, \inf}_{(\mathcal B_j)} (f)
\\[1ex]
\subseteq \WF ^{\, \inf}_{(\mathcal C_j)} (\op (a)f)\ttbigcup
\Char _{(\omega _0)}(a).
\end{multline}
\end{tom}

\par

\begin{rem}\label{remhormWFsets}
We note that many properties valid for the wave-front sets of Fourier
Lebesgue type also hold for wave-front sets in the present
section. For example, the conclusions in Remark \ref{rho0} and Theorem
\ref{hypoellthm} hold for wave-front sets of sup- and inf-types.
\end{rem}

\par

There are (somewhat technical) generalizations of Theorems
\ref{mainthm2} and \ref{mainthm2}$'$ to
pseudo-differential operators with symbols in $S^{(\omega _0)}_{\rho
,\delta}$, when $0\le \delta <\rho \le1$. For example, when
generalizing Theorem \ref{mainthm2}$'$ to $\delta \ge 0$ the key
estimate \eqref{F1aest} needs to be modified into
$$
|(\mathscr F_1a)(\xi ,\eta )|\le C_N\eabs \xi ^{-N}\eabs \eta ^{\delta
N}\omega _0(\eta ).
$$
This in turn implies that in \eqref{wavefrontemb1}$'$ and
\eqref{wavefrontemb1}$''$, the array $(\mathcal C_j)$ on the left-hand
(right-hand) side embeddingsshould be replaced by $(\mathscr
FL^{q_j}_{(\omega _{j,-N_j}/\omega _0)})$, and the right-hand side
embeddings by $(\mathscr FL^{q_j}_{(\omega _{j,N_j}/\omega
_0)})$. Here
$$
\omega _{j,s}(x,\xi )=\omega _j(x,\xi )\eabs \xi ^s,
$$
and
\begin{equation}\label{Njdef}
N_j=\delta (t_j+C),
\end{equation}
where $C$ should be chosen large enough and depends on the order of
the involved distribution $f$ in \eqref{wavefrontemb1}$'$ and
\eqref{wavefrontemb1}$''$ and the dimension $d$, and $t_j$ is chosen
such that the inequality
\begin{equation}\label{t2jdef}
\omega _{j}(x,\xi _1+\xi _2)\le C\omega _{j}(x, \xi _1)\eabs {\xi
_2}^{t_j},
\end{equation}
should hold.

\par

The following generalization of Theorem \ref{mainthm2}$'$ is
obtained by modifying the proof of Proposition \ref{keyprop2AA} and
Theorem \ref{mainthm2}$'$. The details are left for the reader.

\par

\begin{thm}\label{mainthm44}
Let $0\le \delta <\rho \le 1$, $\omega _j\in \mathscr P(\rr
{2d})$, $\omega _0 \in \mathscr P_{\rho ,\delta}(\rr
{2d})$, $a\in S^{(\omega _0)}_{\rho ,0} (\rr {2d})$, $f\in \mathscr
S'(\rr d)$ and $q_j\in [1,\infty ]$ for $j\in J$. Also let $N_j$ be
given by \eqref{Njdef} with $C$ only depending on the order of $f$ and
the dimension $d$,
\begin{equation*}
(\mathcal B_j) \equiv (\mathcal B_j)_{j\in J}\quad \text{and}\quad
(\mathcal C_j^{\pm}) \equiv (\mathcal C_j^{\pm})_{j\in J},
\end{equation*}
where
$$
\mathcal B_j = \mathscr FL^{q_j}_{(\omega _j)} = \mathscr
FL^{q_j}_{(\omega _j)}(\rr d) \quad \text{and}\quad  \mathcal C_j
^{\pm}= \mathscr FL^{q_j}_{(\omega _{j,\pm N_j}/\omega _0)}=\mathscr
FL^{q_j}_{(\omega _{j,\pm N_j}/\omega _0)}(\rr d).
$$
Then
\begin{multline}\label{wavefrontemb2}
\WF ^{\, \sup}_{(\mathcal C^- _j)} (\op (a)f) \subseteq
\WF ^{\, \sup}_{(\mathcal B_j)} (f)
\\[1ex]
\subseteq \WF ^{\, \sup}_{(\mathcal C^+ _j)} (\op (a)f)\ttbigcup
\Char _{(\omega _0)}(a),
\end{multline}
and
\begin{multline}\tag*{(\ref{wavefrontemb2})$'$}
\WF ^{\, \inf}_{(\mathcal C^- _j)} (\op (a)f) \subseteq
\WF ^{\, \inf}_{(\mathcal B_j)} (f)
\\[1ex]
\subseteq \WF ^{\, \inf}_{(\mathcal C^+ _j)} (\op (a)f)\ttbigcup
\Char _{(\omega _0)}(a),
\end{multline}
provided $C$ in \eqref{Njdef} is chosen large enough.
\end{thm}

\par

A combination of Remark \ref{remhormWFsets} and Theorem
\ref{mainthm44} now gives the following result concerning wave-front
sets of H{\"o}rmander type.

\par

\begin{thm}\label{thmclassicWFs}
Let $0\le \delta <\rho \le 1$ and $\omega _0\in \mathscr
P_{\rho,\delta }(\rr {2d})$. For every $f\in \mathscr S'(\rr d)$ and
$a\in S^{(\omega _0)}_{\rho ,\delta}(\rr {2d})$ it holds
$$
\WF  (\op (a)f) \subseteq \WF (f) \subseteq \WF  (\op (a)f)\ttbigcup
\Char _{(\omega _0)}(a).
$$
In particular, if in addition $a$ is elliptic with respect to $\omega
_0$, then
$$
\WF  (\op (a)f) = \WF (f).
$$
\end{thm}

\par

\section{Wave front sets with respect to modulation
spaces}\label{sec6}

\par

In this section we define wave-front sets with respect to modulation
spaces, and show that they coincide with wave-front sets of Fourier
Lebesgue types. In particular, any property valid for wave-front set
of Fourier Lebesgue type carry over to wave-front set of modulation
space type.

\par

Let $\phi \in \mathscr S(\rr d )\setminus 0$, $\omega \in
\mathscr P(\rr {2d})$, $\Gamma \subseteq \rr d\setminus 0$ be an open
cone and let $p,q\in [1,\infty ]$. For any $f\in \mathscr S'(\rr d)$
we set
\begin{multline}\label{modseminorm}
|f|_{\mathcal B(\Gamma )} = |f|_{\mathcal B(\phi ,\Gamma )}
\equiv
\Big ( \int _{\Gamma} \Big ( \int _{\rr {d}} | V_\phi f(x,\xi )\omega
(x,\xi )|^p\, dx\Big )^{q/p}\, d\xi \Big )^{1/q}
\\[1ex]
\text{when}\quad \mathcal B=M^{p,q}_{(\omega )}=M^{p,q}_{(\omega
)}(\rr d)
\end{multline}
(with obvious interpretation when $p=\infty$ or $q=\infty$). We
note that $|\cdo |_{\mathcal B(\Gamma )}$ defines a semi-norm
on $\mathscr S'$ which might attain the value $+\infty$. If $\Gamma
=\rr d\setminus 0$, then $|f|_{\mathcal B(\Gamma )} = \nm
f{M^{p,q}_{(\omega )}}$. We also set
\begin{multline}\label{modseminorm2}
|f|_{\mathcal B(\Gamma )} = |f|_{\mathcal B(\phi ,\Gamma )}
\equiv
\Big ( \int _{\rr {d}} \Big ( \int _{\Gamma} | V_\phi f(x,\xi )\omega
(x,\xi )|^q\, d\xi \Big )^{p/q}\, dx \Big )^{1/p}
\\[1ex]
\text{when}\quad \mathcal B=W^{p,q}_{(\omega )}=W^{p,q}_{(\omega
)}(\rr d)
\end{multline}
and note that similar properties hold for this semi-norm comparing to
\eqref{modseminorm}.

\par

Let $\omega \in \mathscr P(\rr {2d})$, $p,q\in [1,\infty ]$, $f\in
\mathscr D'(X)$, and let
$\mathcal B=M^{p,q}_{(\omega )}$ or $\mathcal B=W^{p,q}_{(\omega
)}$. Then $\Theta _{\mathcal B}(f)$, $\Sigma
_{\mathcal B}(f)$ and the wave-front set $\WF
_{\mathcal B}(f)$ of $f$ with respect to the modulation space
$\mathcal B$ are defined in the same way as in Section
\ref{sec2}, after replacing the semi-norms of Fourier Lebesgue types in
\eqref{notconv} with the semi-norms in \eqref{modseminorm} or
\eqref{modseminorm2}.

\par

The following result shows that wave-front sets of Fourier Lebesgue
 and modulation space types agree with each others.

\par

\begin{thm}\label{wavefrontsequal}
Let $p,q\in [1,\infty ]$, an open set $X\subseteq \rr d$,
$\omega \in \mathscr P(\rr {2d})$, $\mathcal B=\mathscr FL^q_{(\omega
)}(\rr d)$, and let $\mathcal C=M^{p,q}_{(\omega )}(\rr d)$ or
$W^{p,q}_{(\omega )}(\rr d)$. If $f\in \mathscr D'(X)$, then
\begin{equation}\label{WFidentities2}
\WF _{\mathcal B}(f)= \WF _{\mathcal C}(f).
\end{equation}
In particular, $\WF  _{\mathcal C}(f)$ is independent of $p$ and
$\phi \in \mathscr S(\rr d)\setminus 0$ in \eqref{modseminorm} and
\eqref{modseminorm2}.
\end{thm}

\par

\begin{proof}
We only consider the case $\mathcal C = M^{p,q}_{(\omega )}$. 
The case $\mathcal C = W^{p,q}_{(\omega )}$ follows by similar arguments and is
left for the reader. We may also assume that $f\in \mathscr E'(\rr
d)$ and that $\omega (x,\xi )=\omega (\xi )$, since the statements only
involve local assertions.

\par

First we prove that $\WF _{\mathcal C}(f)$ is independent of $\phi \in
\mathscr S(\rr d)\setminus 0$. Therefore assume that $\phi ,\phi _1\in
\mathscr S\setminus 0$ and let $|\cdo |_{\mathcal C_1(\Gamma )}$ be
the semi-norm in \eqref{modseminorm} after $\phi$ has been replaced by
$\phi _1$. Let $\Gamma _1$ and $\Gamma _2$ be open cones in $\rr d$
such that $\overline {\Gamma _2}\subseteq \Gamma _1$. The asserted
independency of $\phi$ follows if we prove that
\begin{equation}\label{est2.6}
|f|_{\mathcal C(\Gamma _2)} \le C(|f|_{\mathcal C_1(\Gamma _1)}+1),
\end{equation}
for some constant $C$.

\par

When proving \eqref{est2.6} we shall mainly follow the proof of
\eqref{cuttoff1}. Let $v\in \mathscr P$ be chosen such that $\omega$
is $v$-moderate, let
$$
\Omega _1=\sets {(x,\xi )}{\xi \in \Gamma _1}\subseteq \rr {2d}\quad
\text{and}\quad \Omega
_2=\complement \Omega _1\subseteq \rr {2d},
$$
with characteristic functions $\chi _1$ and $\chi _2$ respectively,
and set
$
F_k(x,\xi )=|V_{\phi _1}f(x,\xi )|\omega (\xi )\chi _k(x,\xi )$,
$k=1,2,$ and $  G=|V_\phi \phi _1(x,\xi )|v(\xi ).$
By Lemma \ref{stftproperties}, and the fact that $\omega$ is
$v$-moderate we get
$$
|V_\phi f(x,\xi )\omega (x,\xi )|\le C \big ( (F_1+F_2)*G\big )(x,\xi
),
$$
for some constant $C$, which implies that
\begin{equation}\label{fJ1J2igen}
|f|_{\mathcal C(\Gamma _2)} \le C(J_1+J_2),
\end{equation}
where
$$
J_k = \Big (\int _{\Gamma _2} \Big (\int |(F_k*G)(x,\xi )|^p\, dx\Big
)\, d\xi \Big )^{1/q}, \quad k=1,2.
$$

\par

Since $G$ turns rapidly to zero at infinity, Young's inequality
gives
\begin{equation}\label{J1estimateB}
J_1\le \nm {F_1*G}{L^{p,q}_1}\le \nm G{L^1}\nm {F_1}{L^{p,q}_1}
=C|f|_{\mathcal C_1(\Gamma _1)},
\end{equation}
where $C=\nm G{L^1}<\infty$.

\par

Next we consider $J_2$. By Lemma \ref{stftcompact} and the proof of
\eqref{cuttoff1}, it follows that for every $N\ge 0$ there are
constants $C_N$ such that
$$
F_2(x,\xi ) \le C_N\eabs x^{-N}\eabs \xi ^{N_0},\quad \text{and}\quad
\eabs {\xi -\eta }^{-2N}\le C_N\eabs \xi ^{-N}\eabs \eta ^{-N}
$$
when $\xi \in \Gamma _2$ and $\eta \in \complement \Gamma _1$. This in
turn implies that for every $N\ge 0$ there are constants $C_N$ such
that
$$
(F_2*G)(x,\xi )\le C_N\eabs x^{-N}\eabs \xi ^{-N},\quad \xi \in \Gamma
_2.
$$
Consequently, $J_2<\infty$. The estimate \eqref{est2.6} is now a
consequence of \eqref{fJ1J2igen}, \eqref{J1estimateB} and the fact
that $J_2<\infty$. This proves that $\WF _{\mathcal C}(f)$ is
independent of $\phi \in \mathscr S(\rr d)\setminus 0$.

\medspace

In order to prove \eqref{WFidentities2} we assume from now on that
$\phi$ in \eqref{modseminorm} has compact support. We choose
$p_0,p_1\in [1,\infty ]$ such
that $p_0\le p$ and $1/p_1+1/p_0=1+1/p$, and we set $\mathcal C_0
=M^{p_0,q}_{(\omega )}$. The result follows if we prove
\begin{gather}
\Theta _{\mathcal C_0}(f)\subseteq \Theta _{\mathcal B}(f)\subseteq
\Theta _{\mathcal C}(f)\quad \text{when}\ p_0=1,\
p=\infty ,\label{Thetaest1}
\intertext{and}
\Theta _{\mathcal C}(f)\subseteq \Theta _{\mathcal
C_0}(f).\label{Thetaest2}
\end{gather}

\par

We start to prove \eqref{Thetaest1}. We have
\begin{align*}
|f|_{\mathcal B(\Gamma )} &\le
C_1\Big ( \int _{\Gamma} |\widehat f(\xi )\omega (\xi )|^{q}\, d\xi
\Big )^{1/q}
\\[1ex]
&= C_2\Big ( \int _{\Gamma} |\mathscr F \Big ( f \int _{\rr {d}}\phi
(\cdot - x)\,  dx\Big ) (\xi ) \omega (\xi )|^{q}\, d\xi \Big
)^{1/q}
\\[1ex]
& C_2\leq \Big ( \int _{\Gamma} \Big (  \int _{\rr {d}} |\mathscr F (
f \phi (\cdot -x)) (\xi ) \omega (\xi )|\,  dx \Big )^{q}\, d\xi \Big
)^{1/q}
\\[1ex]
&\leq C_3 \Big ( \int _{\Gamma} \Big (  \int _{\rr {d}} |V_{\phi}
f(x,\xi )\omega (\xi )|\,  dx \Big )^{q} \, d\xi \Big )^{1/q} = C_3
|f|_{\mathcal C_0(\Gamma )}
\end{align*}
for some constants $C_1$, $C_2$ and $C_3$. This gives the first
inclusion in \eqref{Thetaest1}.

\par

Next we prove the second inclusion in \eqref{Thetaest1}. Let $K
\subseteq \rr d$ be compact and chosen such that $V_\phi f(x,\xi )=0$
outside $K$. This is possible since both $f$ and $\phi$ have compact
supports. Then
\begin{multline*}
|f|_{\mathcal C(\Gamma _2)} = \Big ( \int _{\Gamma _2}
\sup_{x \in K} | V_\phi f(x,\xi )\omega (x,\xi ) |^{q}\, d\xi \Big
)^{1/q}
\\[1ex]
\leq C_1 \Big ( \int _{\Gamma _2} \sup_{x \in  \rr d}
|( |\widehat f| * | \mathscr F (\phi (\cdot - x)) | ) (\xi) \omega
(\xi )|^{q}\, d\xi \Big )^{1/q}
\\[1ex]
= C_1  \Big ( \int _{\Gamma _2}
|( |\widehat f| * | \widehat \phi  | ) (\xi) \omega (\xi )|^{q}\,
d\xi \Big )^{1/q}
\\[1ex]
\leq C_2  \Big ( \int _{\Gamma _2} \big ( ( |\widehat f \cdot  \omega
| * |\widehat \phi \cdot v| ) (\xi) \big )^{q} \, d\xi \Big )^{1/q},
\end{multline*}
for some positive constants $C_1$, $C_2$ and $C_3$. By combining the
latter estimates with \eqref{cuttoff1} and
its proof it now follows that for each $N\ge 0$ there are constants
$C_N$ such that \eqref{cuttoff1} holds with $\fy =1$ and $\mathcal
B_0=M^{p,q}_{(\omega )}$. Since $f$ has compact support it follows
that the right-hand side of \eqref{cuttoff1} is finite when $|
f|_{\mathcal B(\Gamma _1)}<\infty$, provided $N$ is chosen large
enough. This proves \eqref{Thetaest1}.

\par

It remains to prove \eqref{Thetaest2}. Let $K$ be as above. By
H{\"o}lder's inequality we get
\begin{multline*}
|f|_{\mathcal C_0(\Gamma )} = \Big ( \int _{\Gamma} \Big
( \int _{\rr {d}} | V_\phi f(x,\xi )\omega (x,\xi )|^{p_0}\, dx\Big
)^{q/p_0}\, d\xi \Big )^{1/q}
\\[1ex]
\le C_K\Big ( \int _{\Gamma} \Big (
\int _{\rr {d}} | V_\phi f(x,\xi )\omega (x,\xi )|^{p}\, dx\Big
)^{q/p}\, d\xi \Big )^{1/q} = C_K|f|_{\mathcal C(\Gamma )}.
\end{multline*}
This gives \eqref{Thetaest2}, and the proof is complete.
\end{proof}

\par

\begin{cor}
Let $ f \in  \mathscr E'(\rr d)$, and let
$\mathcal B$ be equal to $\mathscr FL^{q}_{(\omega)}$,
$M^{p,q}_{(\omega )}$ or $W^{p,q}_{(\omega )}$. Then
$$
f \in \mathcal B \quad
\Longleftrightarrow \quad
\WF _{\mathcal B }(f) =\emptyset.
$$
\end{cor}
In particular, we recover Theorem 2.1 and Remark 4.4 in \cite{RSTT}.

\par

\section{Wave-front sets and pseudo-differential
operators with non-smooth symbols}\label{sec7}

\par

In this section we generalize wave-front results in Section \ref{sec3}
to pseudo-differential operators with symbols in $\mho _{(\omega
)}^{s,\rho}(\rr {2d})$ (see Definition \ref{symbolclass}).
In order to state the results we use the convention
\begin{equation}\label{conven1}
(\vartheta _1,\vartheta _2)\lesssim (\omega _1,\omega _2)
\end{equation}
when $\omega _j,\vartheta _j \in \mathscr P(\rr {2d})$ for $j=1,2$
satisfy $\vartheta _j\le C\omega _j$ for some constant $C$. We recall that
\begin{equation}\label{t1t2}
\omega _j(x,\xi _1+\xi _2)\le C\omega  _{j}(x, \xi _1)\eabs
{\xi _2} ^{t_j}, \quad j=1,2,
\end{equation}
for some positive constants $C$, $t_1$ and $ t_2 $, which are
independent of $x,\xi _1,\xi _2\in \rr d$. We let
$\omega_{s,\rho}$ be the same as in \eqref{omegasrho} and we use the
notation $\omega \preccurlyeq (\omega _1,\omega _2 )$ when
\eqref{e5.9} holds for some constant $C$.

\par

\begin{thm}\label{mainthm1}
Let $q\in [1,\infty ]$, $\omega _j,\vartheta _j\in \mathscr
P(\rr {2d})$, $\omega \in \mathscr P_{\rho ,0} (\rr {4d})$, $0<\rho
\le 1$, and let $t_j \geq 0$, $j=1,2$ be chosen such that
\eqref{conven1} and \eqref{t1t2} holds. Also let $\mathcal B=\mathscr
FL^q_{(\omega _1)}$ and $\mathcal C=\mathscr FL^q_{(\omega _2)}$.
Moreover, assume that $\omega
_{s,\rho}$ in \eqref{omegasrho} satisfy $\omega _{s,\rho}\preccurlyeq
(\omega _1,\omega _2)$ $\omega _{s,\rho}\preccurlyeq (\vartheta
_1,\vartheta _2)$ for some choices of
$$
s_1 \geq 0,\quad s_2\in \mathbf N,\quad s_3>t_1+t_2+2d, \quad
\mbox{and} \quad s_4\in \mathbf R.
$$
If $a\in \mho _{(\omega )}^{s,\rho}(\rr {2d})$ and
 $f\in M^{\infty}_{(\vartheta _1)}(\rr d)$, then
$$
\WF _{\mathcal C} (\op (a)f)\subseteq \WF _{\mathcal B} (f).
$$
\end{thm}

\par

By Proposition \ref{p5.4}, it follows that $\op (a)f$ in
Theorem \ref{mainthm1} makes sense as an element in $M^\infty
_{(\vartheta _2)}(\rr {d})$, which contains each space
$M^{p,q}_{(\omega _2)}(\rr {d})$.

\par

When proving Theorem \ref{mainthm1}, we shall mainly follow the ideas
in the proof of Theorem \ref{mainthm2}, and prove some preparing
results. The first one of these results can be considered as a
generalization of Proposition \ref{propmain1AA} in the case $\delta
=0$.

\par

\begin{prop}\label{propmain1}
Let  $a\in \mho_{(\omega )}^{s,\rho}(\rr {2d})$, where
$\omega \in \mathscr P_{\rho ,0} (\rr {2d}\oplus \rr {2d})$,
$0\le s_1$, $0\le s_2\in \mathbf Z$, and let $L_a$ be the operator
from $\mathscr S(\rr d)$ to $\mathscr S'(\rr d)$ which is given by
\eqref{Ladef}. Then the following is true:
\begin{enumerate}
\item $L_a =\op (a_0)$, for some $a_0\in \mathscr S'(\rr {2d})$ such
that
\begin{equation}\label{newsymb}
\partial ^{\alpha}_{\xi} a_0(x,\xi ) \in \underset {s_4\ge
0} \ttbigcap M^{\infty ,1}_{(1/\omega _{s,\rho})}(\rr {2d}),
\end{equation}
for all multi-indices $\alpha$ such that $|\alpha |\le 2s_2$;

\vrum

\item if $p,q\in [1,\infty ]$, and $\omega _1,\omega _2\in \mathscr
P(\rr {2d})$ fulfill $\omega _{s,\rho}\preccurlyeq (\omega _1,\omega
_2)$, then the definition of $L_a$ extends uniquely to a continuous
map from $M^{p,q}_{(\omega _1)}(\rr d)$ to $M^{p,q}_{(\omega _2)}(\rr
d)$.
\end{enumerate}
\end{prop}

\par

\begin{proof}
We use the same notations as in the proof of Proposition
\ref{propmain1AA}, with the difference that $s=(s_1,s_2,s_3,s_4)\in
\rr 4$.

\par

(1) From the proof of Proposition \ref{propmain1AA} it follows that
$L_a=\op (b_s)$, where $b_s$ is given by \eqref{bsformula}. By the
support properties of $\fy _1$, $\fy _2$ and $\fy$, it follows from
Proposition 4.3 in \cite{RSTT} that
\begin{align*}
b_s&\in M^{\infty ,1}_{(\nu _{s,t,\rho })}(\rr {3d}),\quad \mbox{for
every}\quad s_4,t\ge 0,
\intertext{where}
\nu _{s,t,\rho } (x,y,\xi ,\zeta ,\eta ,z) &= \omega _{s,\rho}(x,\xi
,\zeta ,z)^{-1}\kappa (x,y,\xi ,\zeta ,\eta ,z)
\\[1ex]
&=\omega (x,\xi ,\zeta ,z)^{-1}\eabs x^{s_4}\eabs \zeta ^{s_3}\eabs
{x-y}^{2s_2}\eabs \xi ^{\rho s_2}\eabs \eta ^{t}\eabs z^{s_1}.
\end{align*}
Here
$$
\kappa (x,y,\zeta ,\xi ,\eta ,z) = \eabs {z-y}^{2s_2}\eabs \eta ^t
$$

\par

In view of Sections 18.1 and 18.2 in \cite {Ho1} it follows that $\op
(b_s)=\op (a_0)$ when
$$
c_s(x,y,\xi )=e^{i\scal {D_\xi}{D_y}}b_s(x,y,\xi )\quad \text{and}\quad
a_0(x,\xi )=c_s(x,x,\xi ).
$$
We have to prove that $a_0$ is well-defined and fulfills
\eqref{newsymb} when $|\alpha |\le 2s_2$. By Proposition
1.7 in \cite{Toft4} we have
\begin{equation} \label{csts4}
c_s \in M^{\infty ,1}_{(\widetilde \nu _{s,t,\rho })}(\rr {3d}),\quad
\mbox{for every}\quad s_4,t\ge 0,
\end{equation}
where
\begin{gather*}
\widetilde
\nu_{s,t,\rho}(x,y,\xi ,\zeta ,\eta ,z) =
\nu_{s,t,\rho}(x,y-z,\xi -\eta ,\zeta ,\eta ,z)
\\[1ex]
= \omega (x,\xi -\eta ,\zeta ,z)^{-1}\eabs x^{s_4}\eabs \zeta ^{s_3}\eabs
{y-z-x}^{2s_2} \eabs {\xi -\eta }^{2\rho s_2}\eabs \eta ^t\eabs z^{s_1}
\end{gather*}
Since $\omega \in \mathscr P$, the
right-hand side can be estimated by
$$
C\omega _{s,\rho}(x,\xi ,\zeta ,z)^{-1}\eabs
{y-z}^{2s_2} \eabs \eta ^t = C\omega _{s,\rho}(x,\xi ,\zeta
,z)^{-1}\kappa (x,y,\xi ,\zeta ,\eta ,z),
$$
for some constant $C$, provided $s_4$ and $t$ have been replaced by
larger constants if necessary. Since \eqref{csts4} holds for any
$s_4\ge 0$ and $t\ge 0$ we get
$$
c_s \in M^{\infty ,1}_{(1/\omega _{s,t,\rho })}(\rr {3d}),\quad \mbox{for
every}\quad s_4,t\ge 0,
$$
where
$$
\omega _{s,t,\rho } (x,y,\xi ,\zeta ,\eta ,z)
= \omega _{s,\rho} (x,\xi ,\zeta ,z)/\kappa (x,y,\xi ,\zeta ,\eta ,z).
$$

\par

From the fact that
$$
\sup _{z,\eta }\Big ( \big ( \inf _{x,\zeta} \kappa (x,y,\xi ,\zeta
,\eta ,z) \, \big )^{-1}\Big )= 1<\infty ,
$$
it follows now by  Theorem 3.2 in \cite{To8} that $a_0(x,\xi
)=c_s(x,x,\xi )$ is well-defined and belongs to $M^{\infty
,1}_{(1/\omega _{s,\rho})}(\rr {2d})$, for each $s_4\ge 0$. This
proves \eqref{newsymb} in the case $\alpha =0$.

\par

If we let $b_{s,\alpha }$ for $|\alpha |\le 2s_2$ here above be
defined by
$$
b_{s,\alpha }(x,y,\xi ) = \partial ^\alpha _\xi a_0(x,\xi )\fy_1(x)\fy
_2(y),
$$
then
$$
\op (\partial ^\alpha a_0) = \op (b_{s,\alpha}),\quad |\alpha |\le
2s_2.
$$
By similar arguments as in the first part of the proof it follows that
$\partial ^\alpha a_0 \in M^{\infty ,1}_{(1/\omega _{s,\rho})}(\rr
{2d})$, for each $s_4\ge 0$.
The details are left for the reader. This proves (1), and the
assertion (2) is an immediate consequence of (1) and Proposition
\ref{p5.4}. The proof is complete.
\end{proof}

\medspace

Next we consider properties of the wave-front set of $\op (a)f$ at a
fixed point when $f$ is concentrated to that point. In these
considerations it is natural to assume that involved weight functions
satisfy
\begin{equation}\label{weightcond3}
\begin{alignedat}{2}
\omega _j(x,\xi ) &=\omega _j(\xi ),&\qquad \vartheta
_j(x,\xi ) &=\vartheta _j(\xi ),\quad j=1,2,
\\[1ex]
\omega (0,\xi ,\zeta ,z) &\le C\frac {\omega _1(\xi +\zeta )}{\omega
_2(\xi )},& \qquad \omega (0,\xi ,\zeta ,z)&\le C\frac {\vartheta
_1(\xi +\zeta )}{\vartheta _2(\xi )},
\end{alignedat}
\end{equation}
and we set
\begin{equation}\label{extweight1}
\omega _s(x,\xi, \zeta ,z) = \eabs x^{-s_4}\eabs \zeta ^{-s_3}\omega
(0,\xi ,\zeta ,z), \quad s \in \rr 4.
\end{equation}
We also note that
\begin{equation} \label{extweight4}
\omega _j(\xi _1+\xi _2)\le C\omega  _{j}(\xi _1)\eabs
{\xi _2} ^{t_j},\quad j=1,2,
\end{equation}
for some real numbers $t_1$ and $t_2$.

\par

\begin{prop}\label{keyprop2}
Let $q\in [1,\infty]$, and let $\omega _s\in \mathscr P(\rr{4d})$,
$\omega \in \mathscr P(\rr {3d})$, $\omega _{j},\vartheta _{j}\in
\mathscr P(\rr d)$, $j=1,2$,  fulfill $(\vartheta _{1},\vartheta
_{2})\lesssim (\omega _1,\omega _2)$,
\eqref{weightcond3}--\eqref{extweight4}, $s_4>d$ and
$$
s_3>t_1+t_2+2d.
$$
If  $a\in M^{\infty ,1}_{(1/\omega
_s)}(\rr {2d})$ and $f\in M^\infty _{(\vartheta _{1})}(\rr d)\bigcap
\mathscr E'(\rr d)$ then {\rm{(1)}} and {\rm{(2)}} in Proposition
\ref{keyprop2AA} holds for $\mathcal B=\mathscr FL^q_{(\omega _1)}$ and $\mathcal
C=\mathscr FL^q_{(\omega _2)}$.
\end{prop}

\par

\begin{proof}
As for the proof of Proposition \ref{keyprop2AA}, we only prove the
result for $q<\infty$. The slight modifications to
the case $q=\infty$ are left for the reader. We also use similar
notations as in the proof of Proposition \ref{keyprop2AA}.

\par

Let $\phi _1,\phi _2\in C_0^\infty (\rr d)$ be such that
$$
\int _{\rr {d}} \phi _1(x)\, dx =1,\quad \phi _2(0)=(2\pi )^{-d/2},
$$
and set $\phi =\phi _1\otimes \phi _2$. It follows from Fourier's inversion
formula that
\begin{equation*}
\mathscr F_1a(\xi ,\eta ) = \iint (V_\phi a)(x,\eta ,\xi
,z)e^{i\scal z\eta }\, dxdz .
\end{equation*}
By Remark \ref{coorb} it follows that
\begin{multline*}
|(\mathscr F_1a)(\xi -\eta, \eta )\omega _2(\xi )| \le C_1\iint
|(V_\phi a)(x,\eta, \xi -\eta  ,z)\omega _2(\xi )|\, dxdz
\\[1ex]
\le C_2 \iint \big ( \sup _{\eta }  \big ( \sup _{x,\xi \in \rr d}|V_\phi
a(x,\xi ,\zeta ,z)\, \omega (x,\xi ,\zeta, z)^{-1}| \big ) \big )
\times
\\[1ex]
\times
\eabs x^{-s_4}\eabs{\xi- \eta} ^{-s_3}\omega (0, \eta, \xi
-\eta , z )\omega _2(\xi )\, dxdz
\\[1ex]
\le C_{3} \nm a{\widetilde M_{(1/\omega _s)}}\eabs{\xi- \eta}
^{-s_3} (\sup _{z} \omega (0,\eta, \xi -\eta , z))\omega _2(\xi )
\\[1ex]
\le C_4 \nm {a}{M^{\infty ,1}_{(1/\omega _s)}}  \eabs{\xi - \eta} ^{-s_3}
(\sup _{z} \omega (0, \eta ,\xi -\eta , z))\omega _2(\xi )
\\[1ex]
\le C_5 \nm {a}{M^{\infty ,1}_{(1/\omega _s)}}\eabs{\xi - \eta}
^{-s_3}\omega _1(\eta ),
\end{multline*}
for some constants $C_1,\dots ,C_5$. The result now follows by similar
arguments as in the proof of Proposition \ref{keyprop2AA}, after
replacing \eqref{keyprop2AAest} with the latter estimates. The details
are left for the reader, and the proof is complete.
\end{proof}

\par

\begin{proof}[Proof of Theorem \ref{mainthm1}]
Let $\fy \in C_0^\infty (\rr d)$ be such that $\fy =1$
in a neighborhood of $x_0$, and let $\fy _2=1-\fy$.
If $(x_0,\xi _0)\notin \WF  _{\mathcal B}(f)$,
then it
follows from Proposition \ref{propmain1} that
$$
(x_0,\xi _0)\notin \WF  _{\mathcal C}(\op (a)(\fy
_2f)).
$$
Since
$$
\WF  _{\mathcal C}(\op (a)f)
\subseteq \WF  _{\mathcal
C}(\op (a)(\fy f))\ttbigcup \WF  _{\mathcal
C}(\op (a)(\fy _2f)),
$$
the result follows if we prove that
\begin{equation}\label{keystep1}
(x_0,\xi _0)\notin \WF  _{\mathcal C}(\op (a_0)(\fy f)),
\end{equation}
where $a_0(x,\xi )=\fy (x)a(x,\xi )$.

\par

We may assume that $\omega (x,\xi ,\zeta ,z)$, $\omega _j(x,\xi )$ and
$\vartheta _j(x,\xi )$ are independent of the $x$ variable when
proving \eqref{keystep1}, since both $\fy f$ and $a_0$ are compactly
supported with respect to the $x$ variable. The
result is now an immediate consequence of Proposition
\ref{keyprop2}.
\end{proof}

\par

\begin{rem}\label{otherpseudoBB}
Let  $\mho ^{s,\rho ,t}_{(\omega )}(\rr
{2d})$ be as $\mho ^{s,\rho }_{(\omega )}(\rr {2d})$, after
$\omega _{s,\rho}(x,\xi ,\zeta ,z)$ has been replaced by
$$
\omega _{s,t,\rho} (x,\xi ,\zeta ,z) = \omega _{s,\rho}(x+tz,\xi
+t\zeta ,\zeta ,z), \quad t\in \mathbf R,
$$
in the definition of $\mho ^{s,\rho }_{(\omega )}(\rr {2d})$. Then it
follows from Proposition 1.7 in \cite{Toft4} that if $a\in \mho
^{s,\rho }_{(\omega )}(\rr {2d})$, then Theorem \ref{mainthm1}
remains valid after $\omega (x,\xi ,\zeta ,z)$ has been replaced
by $\omega (x+tz,\xi +t\zeta ,\zeta ,z)$ and $\op (a)$ has been
replaced by $\op _t(a)$.
\end{rem}

\end{document}